 \documentclass[preprint,hidelinks,authoryear]{elsarticle}
\usepackage{amsfonts}
 \usepackage{amsmath} 
 \usepackage{amsthm}
 \usepackage[utf8]{inputenc}  % UTF-8 input encoding
\usepackage[T1]{fontenc}     % Type1 fonts
\usepackage{lmodern}         % Improved Computer Modern font
\usepackage{microtype}       % Better handling of typo
\usepackage[scaled]{helvet} % Scale Helvetica
\usepackage[english]{babel}  % Hyphenation
\usepackage{graphicx}        % Import graphics
\usepackage[bookmarks=false]{hyperref}            % Better handling of references in PDFs
\usepackage{comment}         % To comment out blocks of text
\biboptions{sort&compress}   % Sort and compress citations
\usepackage{xcolor}
    \definecolor{myred}{RGB}{187, 85, 102}
    \definecolor{myblue}{RGB}{0, 68, 136}
\usepackage{enumitem}
\usepackage{dsfont}

\usepackage{tikz}
\usepackage{graphicx,rotating}
\usepackage[all]{xy}

\usepackage{algorithm} 
\usepackage{algorithmicx}  
\usepackage{algpseudocode}
\usepackage{subcaption}
\usepackage{indentfirst,float}
\usepackage[output-decimal-marker={,}]{siunitx}
\usepackage{tikz,tikz-inet,pgf,pgfplots}

\newtheorem{Thm}{Theorem}[section]
\newtheorem{Lemma}[Thm]{Lemma}
\newtheorem{Cor}[Thm]{Corollary}
\newtheorem{Prop}[Thm]{Proposition}

\theoremstyle{plain}
\newtheorem{theorem}{Theorem}[section]

\theoremstyle{definition}
\newtheorem{definition}[theorem]{Definition}
\newtheorem{example}[theorem]{Example}

\theoremstyle{remark}
\newtheorem{remark}[theorem]{Remark}

\journal{Journal of Mathematical Analysis and Applications}

\begin{document}

\begin{frontmatter}



\title{Symmetry, Scaling, and Optimal Time–Frequency Concentration:
Minimising the Heisenberg Uncertainty in Piecewise‑Polynomial and Wavelet Dictionaries}

\author{Lorenzo de Leonardis\fnref{label1}}
\author{Alessandro Mazzoccoli\fnref{label1}}
\author{\fnref{label1} Pierluigi Vellucci*}
\fntext[label1]{Department of Economics, Roma Tre University, Via Silvio D’Amico 77, 00145 Rome, Italy
\\ *Corresponding author}

\begin{abstract}
In this work, we introduce a hierarchy of function classes defined on a fixed compact interval, along with tailored uncertainty operators. We establish key properties of the associated uncertainty product, showing that it is invariant under scale and translation transformations. Notably, we prove that the infimum of the uncertainty within the asymmetric class is attained in the even subclass. Within two specific wavelet dictionaries, we identify the tent function as the unique minimiser of the time-frequency uncertainty, achieving a value of $U = \frac{3}{10}$. Additionally, we analyse the family of $p$-fold self-convolutions of the rectangle function, $\operatorname{rect}^{\{p\}}$, demonstrating that the uncertainty decreases monotonically towards the Heisenberg bound $ \frac{1}{4} $ as $p \to \infty$. These findings unify and explain various empirical observations from the literature on adaptive wavelet design and Gabor frame stability, and suggest a principled approach to constructing dictionaries with provably optimal joint localisation properties.
\end{abstract}




\end{frontmatter}

%\tableofcontents



\section{Introduction}\label{sec:intro}

The uncertainty principle is a pervasive theme in harmonic analysis, expressing that a nontrivial function cannot be sharply localized in both the original and frequency domains.  In practical terms, improving spatial (or time) concentration inevitably worsens frequency (or spectral) concentration, and vice versa.  This trade-off underlies many classical results, for example, a function and its Fourier transform cannot both have compact support, and the Heisenberg inequality quantifies a lower bound on the product of their effective widths.  Such principles have far-reaching consequences: they limit the resolution of time-frequency analysis and motivate the search for representations that optimally balance localization in dual domains.

One famous manifestation is the Balian–Low Theorem (BLT), which in essence states that time–frequency concentration is incompatible with an orthonormal basis of $L^2(\mathbb{R})$ generated by time-frequency shifts of a single function \cite{battle1988heisenberg}.  Originally proved in the context of Gabor (windowed Fourier) bases, the BLT implies that no nonredundant Gabor system can use a window function that is well-localized in both time and frequency.  Intuitively, if a window $g(t)$ is very sharply localized and its Fourier transform $\hat g(\omega)$ is also well-localized, then $\{e^{2\pi i m t}g(t-n)\}_{n,m\in\mathbb{Z}}$ cannot form an orthonormal basis of $L^2$ – some redundancy or poor decay must intervene.  Recent research has strengthened and generalized this theorem, highlighting its fundamental nature.  For example, \cite{caragea2023balian,caragea2021quantitative} proved a version of Balian–Low for arbitrary closed subspaces of $L^2$, showing that even if one only seeks a Gabor basis for a subspace (not all of $L^2$), severe localization constraints persist. Moreover, the scope of BLT has expanded beyond the real line: \cite{enstad2020balian} extended the theorem to general locally compact abelian groups (and even certain vector bundles), demonstrating that the incompatibility between joint localization and completeness is a universal phenomenon not tied to $\mathbb R$ alone.

These developments underscore that the inability to achieve simultaneous excellent time and frequency localization is intrinsic to a wide class of harmonic-analytic systems.

Another strain of the uncertainty principle focuses on quantitative measures of concentration.  In the context of periodic functions (or functions on the circle), \cite{breitenberger1985uncertainty} introduced an \emph{uncertainty constant} that quantifies joint localization on the circle (essentially an angular analog of the usual variance-based uncertainty) by a single number.  An important inequality due to \cite{breitenberger1985uncertainty} gives a lower bound for this constant, again indicating a fundamental limit to simultaneous localization.  Recent work has revisited and refined this result: \cite{lebedeva2017inequality} showed that for a broad class of periodic functions one can sharpen Breitenberger’s lower bound, improving our understanding of the best possible concentration achievable on the circle. In particular, in her 2017 study \cite{lebedeva2017inequality} she proved an inequality that refines the minimal attainable uncertainty constant for periodic functions. Earlier, \cite{goh2000uncertainty} had already shown that for \emph{periodic} scaling functions and wavelets that are uniformly local, regular, and stable, the associated uncertainty products are \emph{uniformly bounded above}, establishing constructive families (including periodic and trigonometric B–splines) that nearly saturate Breitenberger’s bound while keeping the constant finite. Such investigations not only reinforce the principle that one cannot circumvent the uncertainty principle even on compact domains, but also provide tools to identify extremal or “least uncertainty” functions on those domains.

The pursuit of extremal functions or bases that \emph{approach} the uncertainty limit is a theme connecting uncertainty principles with what we might call “optimal representation theory.”  A classical example is the family of prolate spheroidal wave functions, which achieve maximal energy concentration within a time interval for a fixed bandlimit, offering an optimal solution to the time-frequency trade-off. In the same spirit, researchers have sought bases or frames that are optimally localized in both domains.  For instance, \cite{lebedeva2014periodic} constructed a family of Parseval wavelet frames on the torus that achieve \emph{optimal} time-frequency localization in the sense of the Breitenberger constant.  Their work showed that by allowing a little redundancy (frames instead of bases) and carefully designing the wavelet filters, one can minimize the uncertainty constant uniformly across the frame.  This is a prime example of using an optimality criterion (minimizing an uncertainty measure) to guide the construction of a representation. It demonstrated that while a perfect “ideal” localization is impossible, one can still design wavelet or Gabor systems that are, in a certain sense, as concentrated as theoretically allowed. These optimally localized frames provide practical tools in signal analysis (where one desires sharp time and frequency localization) and, at the same time, illuminate the boundary between the possible and impossible in harmonic analysis. \textcolor{black}{Complementary to this perspective, recent progress pinpoints the Gabor frame set of compactly supported atoms-most notably B-splines-by proving sharp obstructions at (or near) critical density and by giving new sufficient conditions for frame generation (\cite{ghosh2025gabor,ghosh2025obstructions}).}

A recent paper by \cite{aldahleh2025additive} investigated signal recovery on finite abelian groups, focusing on the conditions under which a signal can be exactly reconstructed from a subset of its Fourier coefficients. Building on classical results, the authors introduced a new perspective based on additive energy, a combinatorial quantity measuring the number of additive relations in a set. They developed an improved form of the uncertainty principle in terms of additive energy, leading to stronger recovery guarantees.

Beyond classical Fourier analysis, the robustness of uncertainty limitations extends to generalized transforms and modern variants. \textcolor{black}{In parallel, optimal sampling results in shift-invariant spaces generated by Meyer scaling functions and by $B$-splines further clarify how spectral localization constrains stable reconstruction, echoing uncertainty-type barriers at the representation level (\cite{selvan2017optimal,selvan2016sampling}). Recent works from 2024–2025 sharpen this picture in wavelet-type settings, establishing quantitative large-sieve/uncertainty bounds for the continuous wavelet transform and new Heisenberg-/Donoho-Stark-/logarithmic principles for Mehler-Fock, metaplectic, and quaternionic wavelet transforms, thereby clarifying optimal joint concentration beyond the Fourier paradigm (\cite{abreu2025donoho,dades2024heisenberg,dades2025new,dar2024n,wang2024benedicks}).}
%Even in quaternionic signal analysis, which replaces complex exponentials by quaternion-valued exponentials, analogous uncertainty principles hold.  For instance, \cite{bahri2016modified} formulated an uncertainty inequality for the two-sided quaternionic Fourier transform and proved that the tightest bound is attained (as in the classical case) by Gaussian-like quaternionic functions.  Such results in Clifford analysis and related fields demonstrate that the conflict between simultaneous time (or spatial) and frequency localization is a structural phenomenon, not an artifact of ordinary Fourier analysis alone.  Whether one works with functions on groups, on manifolds, or in noncommutative settings, some version of the uncertainty principle emerges to dictate fundamental limits on localization.

Given this broad landscape, 
%which spans pure harmonic analysis, 
%(such as the Balian–Low theorem and its generalizations) 
%and applied signal processing,
%(such as sparsity-based uncertainty principles), 
it is natural to ask whether there exists a unifying principle or optimality criterion underlying these various uncertainty phenomena.

In this article, we outline a possible framework.  We show that many of the aforementioned uncertainty inequalities can be seen as necessary consequences of a single, overarching optimality criterion.  Roughly speaking, we consider an abstract functional measuring joint localization, and we characterize the functions (or signal expansions) that would extremize this measure.  Our main result is that if a function were to simultaneously achieve “optimal” spatial localization and “optimal” spectral accuracy in this broad sense, it must satisfy a certain extremal equation.  Crucially, this extremality condition \emph{cannot hold} for any nontrivial function, except in trivial or limiting cases, and from this fact the usual uncertainty principles follow.  In particular, our approach recovers the qualitative conclusions of the Balian–Low Theorem, Breitenberger’s inequality, and related uncertainty relations as corollaries of a single optimality framework.  This offers a new perspective: rather than proving each uncertainty principle separately via ad hoc analytical estimates, we derive them from the contrapositive of an optimal representation property. We believe this perspective not only unifies existing results but also provides insight into why the known “extremal” functions (Gaussians, prolate spheroidal wavefunctions, etc.) occupy their special status, since they are the ones that come closest to the unattainable ideal.

\textcolor{black}{In parallel to our variational viewpoint, recent wavelet-focused results deliver sharp concentration controls, most notably large-sieve and local Lieb-type inequalities for the continuous wavelet transform, which complement the impossibility/minimization mechanisms captured by our functional $U$ (\cite{abreu2025donoho}).}

This paper focuses on dictionaries as introduced in \cite{rivero2023solution}, which are collections of functions in $L^2(\mathbb{R})$ defined as follows.
\begin{definition}
    Let $\gamma = (t, \xi, u) \in \Gamma = \mathbb{R}_{>0} \times \mathbb{R}^2$. A \emph{waveform dictionary} $\mathcal{G}$ is a collection of functions in $L^2(\mathbb{R})$ of the form
$$
G_\gamma(x) = \frac{1}{\sqrt{t}}\, g\left(\frac{x - u}{t}\right) e^{2\pi i \xi x},
$$
where $g \in L^2(\mathbb{R})$ is the \emph{window function} satisfying $\|g\|_{L^2(\mathbb{R})} = 1$, $g(0) \neq 0$, and $\int_\mathbb{R} g(x)\,dx \neq 0$. Each $G_\gamma$ is called a \emph{time-frequency atom}.
\end{definition}
We now highlight the key contributions of this paper.

\paragraph*{Main contributions.}
Motivated by the panorama sketched above, this paper proposes a unified
variational perspective on time–frequency localisation and proves three
new results:

\begin{enumerate}[label=\textbf{C\arabic*}.]
  \item We introduce a scale–translation–invariant \emph{uncertainty
        functional} $U$ on a hierarchy of piece\-wise‑polynomial
        function classes and show that all variance–type uncertainty
        inequalities (Heisenberg, Breitenberger, quantitative Balian–Low)
        follow from the impossibility of attaining the minimum of~$U$.
  \item Inside two wavelet dictionaries
        $\{\mathcal G_{\gamma,n}\}_{n\in\mathbb N}$ and
        $\{\mathcal F_{\gamma,n}\}_{n\in\mathbb N}$ we prove that the
        \emph{tent function} ($n=1$) uniquely minimises $U$, achieving the
        optimal value $U=\frac{3}{10}$; this confirms and refines earlier
        empirical observations.
  \item For the family of $p$‑fold convolutions of the rectangular
        window $\operatorname{rect}^{\{p\}}$ we establish two–sided
        bounds showing that $U(\operatorname{rect}^{\{p\}})\searrow
        1/4$ as $p\to\infty$, hence the classical Heisenberg limit is
        asymptotically attainable within this smooth B‑spline scale. \textcolor{black}{This asymptotic B-spline perspective dovetails with recent descriptions of when spline atoms yield stable Gabor systems and when structural obstructions arise (\cite{ghosh2025gabor,ghosh2025obstructions}).}
\end{enumerate}

These results deliver a coherent explanation of why certain well‑known
atoms (Gaussians, tent functions, high‑order B‑splines) appear
repeatedly as “best localised” in both theory and practice, and they
provide concrete design principles for constructing optimally
concentrated frames in applied harmonic analysis.

\textcolor{black}{\paragraph*{Relation to prior work and positioning.}
Within this framework, our contribution takes a complementary position. Like the works on Gabor systems and the Balian–Low principles, we emphasize the role of invariance: our uncertainty product is explicitly constructed to remain stable under scaling and translation. At the same time, rather than emphasizing limitations of localization theorems, we adopt a constructive perspective, identifying explicit minimizers (the tent function and the self-convolutions of the rectangle) and showing that these approach the Heisenberg limit. In this respect, our findings resonate with the work of \cite{rivero2023solution, lebedeva2017inequality} and \cite{goh2000uncertainty}, who construct dictionaries and wavelets with near-optimal localization properties, while also complementing the conceptual analyses of uncertainty measures (\cite{breitenberger1985uncertainty}) and the structural or combinatorial formulations of localization constraints (\cite{caragea2023balian, enstad2020balian, aldahleh2025additive}).}

\textcolor{black}{Moreover, the recent study of \cite{ghosh2025gabor} on the frame sets of B-splines, Hermite functions, and totally positive functions highlights how explicit constructions can extend the admissible parameter regions for Gabor frames, which is closely related to our constructive approach to minimizers. Similarly, \cite{selvan2017optimal} analysis of stable sampling in shift-invariant spaces via Meyer wavelets provides a complementary perspective, where explicit gap conditions guarantee reconstruction, paralleling our focus on precise extremal functions. Finally, \cite{wang2024benedicks} work on uncertainty principles for the continuous quaternion wavelet transform broadens the landscape of invariance and localization, showing how multidimensional and non-commutative settings admit analogues of the Heisenberg, Benedicks, and Sobolev-type principles, further reinforcing the universality of the phenomena we describe. Along the same lines, the contributions of \cite{dades2024heisenberg, dades2025new} on the Mehler–Fock and free metaplectic wavelet transforms extend classical uncertainty principles, including the Donoho–Stark, Lieb, and logarithmic forms, to novel transform settings, thereby underscoring the robustness of uncertainty inequalities across different analytical frameworks. In a related direction, \cite{dar2024n} introduces the free metaplectic wave packet transform in $L^2(\mathbb{R}^n)$.
}

\textcolor{black}{\paragraph*{Organization of the paper.}
We collect the minimal definitions needed for our statements in Section \ref{sec:maindef} and then present the main results in Section \ref{sec:mainres}; proofs and auxiliary lemmas are deferred to Sections \ref{sec2} and \ref{sec3}.}

\textcolor{black}{\section{Main results}\label{sec:main}
The precise statements of our main theorems require a few basic notions.
For clarity, we collect these definitions in the next section, followed by
the enunciation of our main results.
\subsection{General Definitions}
\label{sec:maindef}
The idea is to work with a set of functions that allows for easy computation of Fourier coefficients, so consider a fixed compact support $ K \subseteq \mathbb{R} $. The most natural choice, then, is to consider functions that are sums of polynomials defined on intervals of $ \mathbb{R} $. In the following, the symbol $ \mathds{1}_{[a_i,b_i)} $ represents the indicator function.
\begin{definition} 
Let $ K = [-a,a] \subseteq \mathbb{R} $ be a compact symmetric set, with $a>0$. Then, we define the family of easy functions as follows:
\begin{equation}
\label{facilisuppK}
\begin{aligned}
    \mathcal{F}_{supp}(K) := \Biggl\{& f(x) = \sum_{i=1}^n P_i(x) \mathds{1}_{[a_i,b_i)}(x),\displaystyle\lim_{x \to a^{-}_i}f(x)=f(a_i), \; \forall i=2,...,n,   \; \\&\displaystyle\lim_{x \to b_n^-}f(x)=f(b_n) < +\infty, [a_i, b_i) \subseteq K, \;\\& -a = a_1 < b_1 = a_2 < \dots < b_{n-1} = a_n < b_n = a, \\
    &P_i(x) \text{ is a polynomial of any degree for each } i, \; \text{for some} \; n \in \mathbb{N} \Biggr\}.
\end{aligned}
\end{equation}
where $\mathds{1}_{A}(*)$ is the indicator function defined as $$\mathds{1}_{A}(x)=\begin{cases} 1 & \mbox{if } \; x \in A \\ 0 & \mbox{otherwise }
\end{cases}\,.$$
We also define three families:
\begin{align}
     \mathcal{F}^+_{supp}(K)&:=\{f \in \mathcal{F}_{supp}(K) : f(x) \ge 0 \; \forall x \in K\}
     \label{funioni+}, \\
     \mathcal{F}^+_{0}(K)&:=\{f \in \mathcal{F}^+_{supp}(K) : f(-a)=f(a)=0 \},
     \label{funzioni+0} \\
     \mathcal{P}^+_{0}(K)&:=\{f \in \mathcal{F}^+_{0}(K) : f \; \text{ is even} \}   .    \label{funzionipari}
\end{align}
\end{definition}
Figure \ref{fig:families} illustrates the four families of functions defined on the compact interval $K = [-2, 2]$, each with distinct properties as described in the caption.
\begin{figure}[hbt!]
    \centering
    \includegraphics[width=0.4\linewidth]{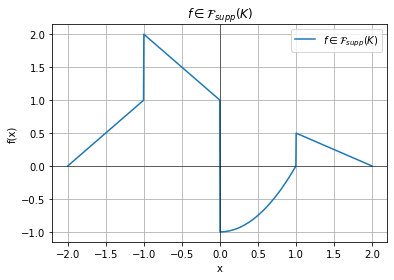}
    \includegraphics[width=0.4\linewidth]{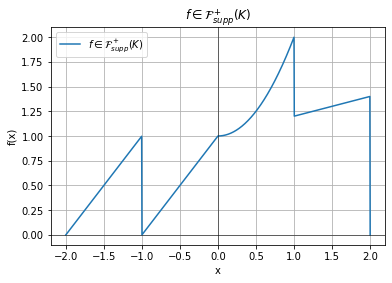}\\
    \includegraphics[width=0.4\linewidth]{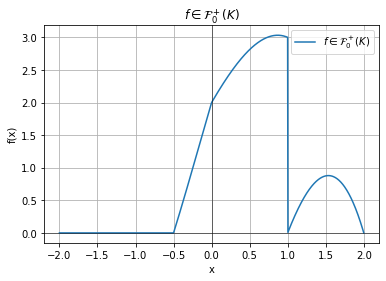}
    \includegraphics[width=0.4\linewidth]{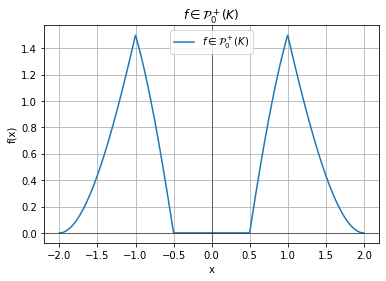}
    \caption{Graphs of the four families of functions defined on the compact interval $ K = [-2, 2] $: (a) $ \mathcal{F}_{supp}(K) $, sum of polynomials defined on disjoint intervals; (b) $ \mathcal{F}^+_{supp}(K) $, non-negative functions on $ K $; (c) $ \mathcal{F}^+_{0}(K) $, non-negative functions on $ K $ with zero values at the endpoints; (d) $ \mathcal{P}^+_{0}(K) $, even non-negative functions on $ K $ with zero values at the endpoints.}
    \label{fig:families}
\end{figure}}

\textcolor{black}{We observe that for all $ f \in \mathcal{F}_{supp}(K) $, by the Weierstrass Theorem, it follows that $ f \in L^1(K, Leb) \cap L^2(K, Leb) $, and therefore the Fourier transform is well-defined for every $ \omega \in \mathbb{R} $
\begin{equation}
\hat{f}(\omega):=\displaystyle\int_{\mathbb{R}}e^{-i\omega x}f(x) \; dx,
\label{trasformatafourierdef}
\end{equation} 
where $L^p(K,Leb)$ is the $L^p$ Banach space of functions whose modulus at $p$ is integrable over $K$ with the Lebesgue measure. Since we are considering always the Lebesgue measure, for brevity we will only indicate the set to be considered and not the measure ($L^p(K,Leb)=L^p(K)$).}

\textcolor{black}{Clearly, the Plancherel Theorem holds (\cite{rudin1987real}), and therefore
\begin{equation}
    \left \| \hat{f} \right \|_{L^2(\mathbb{R})}^2 = 2\pi \left \| f \right \|^2_{L^2(K)}.
    \label{plancherelnorme}
\end{equation}
\begin{definition}[Mean, Variance, and Uncertainty Operators]
    Let $ K \subseteq \mathbb{R} $ be a compact set. Let $ \mathcal{F}_{supp}(K) $ be the set in Eq. \eqref{facilisuppK} and
    $$\hat{\mathcal{F}}_{supp}(K) := \left\{ \hat{f} \mid \hat{f}(\omega) = \int_\mathbb{R} e^{-i\omega x} f(x) \; dx, \; f \in \mathcal{F}_{supp}(K) \right\}$$
    be the space of Fourier transforms.
 \begin{itemize}   
\item[a)] We define the mean operators in space and frequency:
\begin{equation}
\begin{array}{rcl}
    \alpha[\cdot]: \mathcal{F}_{supp}(K) & \to & \mathbb{R} \\
    f & \mapsto & \alpha[f]:= \displaystyle\frac{1}{ \left \| f \right \|^2} \int_K x |f(x)|^2 \; dx
    \label{mediaspazio}
\end{array}
\end{equation}
\begin{equation}
\begin{array}{rcl}
    \beta[\cdot]: \hat{\mathcal{F}}_{supp}(K) & \to & \mathbb{R} \\
    \hat{f} & \mapsto & \beta[\hat{f}]:= \displaystyle\frac{1}{2\pi \left \| f \right \|^2} \int_{\mathbb{R}} \omega |\hat{f}(\omega)|^2 \; d\omega
    \label{mediafrequenza}
\end{array}
\end{equation}
\item[b)] We define the variance operators in space and frequency:
\begin{equation}
\begin{array}{rcl}
    \sigma_x^2[\cdot]: \mathcal{F}_{supp}(K) & \to & \mathbb{R} \\
    f & \mapsto & \sigma_x^2[f]:=\displaystyle \frac{1}{ \left \| f \right \|^2} \int_K (x - \alpha[f])^2 |f(x)|^2 \; dx
    \label{varianzaspazio}
\end{array}
\end{equation}
\begin{equation}
\begin{array}{rcl}
    \sigma_{\omega}^2[\cdot]: \hat{\mathcal{F}}_{supp}(K) & \to & \mathbb{R} \; \cup \{+\infty\} \\
    \hat{f} & \mapsto & \sigma_{\omega}^2[\hat{f}]:=\displaystyle\frac{1}{2\pi \left \| f\right \|^2}\displaystyle\int_{\mathbb R} (\omega-\beta[\hat{f}])^2|\hat{f}(\omega)|^2 \; d\omega
\end{array}
\end{equation}
\item[c)] We define the Uncertainty Operator (Heisenberg's Uncertainty Principle) as follows:
\begin{equation}
    \begin{array}{rcl}
    U[\cdot]: \mathcal{F}_{supp}(K) & \to & \mathbb{R} \; \cup \{+\infty\} \\
    f & \mapsto & U[f]:= \sigma_x^2[f] \cdot \sigma_{\omega}^2[\hat{f}]
    \label{incertezza}
    \end{array}
\end{equation}
\end{itemize}
Analogous definitions hold for $\mathcal{F}^+_{supp}(K)$, $\mathcal{F}^+_{0}(K)$, and $\mathcal{P}^+_{0}(K)$ defined in Eqs. \eqref{funioni+},\eqref{funzioni+0},\eqref{funzionipari}.
\end{definition}
Concerning Theorem \ref{th:infimumevenf}, we define the functions  
\begin{equation}
\label{eq:simmetryfunctions}
f_d(x) :=  
\begin{cases}  
f(x) & \text{if } x \geq 0 \\  
f(-x) & \text{if } x < 0  
\end{cases}  \, , \quad 
f_s(x) :=  
\begin{cases}  
f(x) & \text{if } x < 0 \\  
f(-x) & \text{if } x \geq 0  
\end{cases}    
\end{equation}  
which are, respectively, the right- and left-reflected versions of $ f $ with respect to the $ y $-axis.}

\textcolor{black}{As for Theorems \ref{th:tent1} and \ref{th:tent2}, instead, we need two wavelet dictionaries that will serve as the basis  for our analysis of the uncertainty over the three specific cases considered in this paper:
\begin{align}
\{\mathcal{G}_{\gamma,n}\}_{\gamma \in \Gamma} :&= \left\{ \sqrt{\frac{2n+1}{2t}} \left(1 - \left| \frac{x - u}{t} \right| \right)^n \mathds{1}_{[u - t, u + t]}(x) \, e^{2\pi i \xi x}, \; n \in \mathbb{N} \right\}_{\gamma \in \Gamma} \notag \\
    &= \left\{ g_n\left( \frac{x - u}{t} \right) e^{2\pi i \xi x} \right\}_{\gamma \in \Gamma}
    \label{dizionario1}    
\end{align}
\begin{align}
\label{dizionario2}
    \{\mathcal{F}_{\gamma,n}\}_{\gamma \in \Gamma} :&= \left\{ \sqrt{\frac{(2n+1)(n+1)}{4n^2 t}} \left(1 - \left| \frac{x - u}{t} \right|^n \right) \mathds{1}_{[u - t, u + t]}(x) \, e^{2\pi i \xi x}, \; n \in \mathbb{N}_{> 0} \right\}_{\gamma \in \Gamma} \notag \\
    &= \left\{ f_n\left( \frac{x - u}{t} \right) e^{2\pi i \xi x} \right\}_{\gamma \in \Gamma}
\end{align}
For $ n = 0 $, the function coincides with $ g_0(x) $.}

\textcolor{black}{We define the basic (non-modulated) functions as follows:
\begin{equation}
    g_n(x) := \sqrt{\frac{2n+1}{2t}} \left(1 - |x| \right)^n \, \mathds{1}_{[-1,1]}(x),
    \label{definizioneunog}
\end{equation}
\begin{equation}
    f_n(x) := \sqrt{\frac{(2n+1)(n+1)}{4n^2 t}} \left(1 - |x|^n \right) \, \mathds{1}_{[-1,1]}(x).
    \label{definizioneduef}
\end{equation}
It is straightforward to verify the normalization of both families:
\begin{align}
    \|g_n\|^2 = \int_{u - t}^{u + t} \left| g_n\left( \frac{x - u}{t} \right) \right|^2 \, dx = \frac{2n+1}{2} \int_{-1}^{1} (1 - |x|)^{2n} \, dx = 1, \quad \forall n \in \mathbb{N},
    \label{norm_gn_eng}
\end{align}
\begin{align}
    \|f_n\|^2 &= \int_{u - t}^{u + t} \left| f_n\left( \frac{x - u}{t} \right) \right|^2 \, dx \notag \\
              &= \frac{(2n+1)(n+1)}{4n^2} \int_{-1}^{1} \left(1 - |x|^n \right)^2 \, dx = 1, \quad \forall n \in \mathbb{N}_{>0}.
    \label{norm_fn_eng}
\end{align}
Naturally, both $ g_n(x) $ and $ f_n(x) $ belong to the class $ P^+_0([-1,1]) $ for all $ n \in \mathbb{N} $.}

\textcolor{black}{Lastly, concerning Theorem \ref{th:Urectp}, we consider the rectangular function $ \operatorname{rect}^{\{1\}}(x) := \operatorname{rect}(x) = \mathds{1}_{[-\frac{1}{2}, \frac{1}{2}]}(x) $, and define its $ p $-fold convolution recursively for $ p \ge 2 $ as \cite{schoenberg1946contributions,schoenberg1969cardinal,de2018stability}
$$
\operatorname{rect}^{\{p\}}(x) := \int_{x - \frac{1}{2}}^{x + \frac{1}{2}} \operatorname{rect}^{\{p-1\}}(s) \, ds.
$$
In other words, $ \operatorname{rect}^{\{p\}}(x) $ denotes the $ p $-th convolution of the rectangular function with itself.
}

\textcolor{black}{\subsection{Statements of the main theorems}
\label{sec:mainres}
\begin{Thm}[The Infimum is Attained by Even Functions]
\label{th:infimumevenf}
Let $a>0$ and $f\in\mathcal F_0^+([-a,a])$ be absolutely continuous with $f'\in L^2(-a,a)$. Let $f_s,f_d$ be the left/right reflections of $f$ \emph{about its barycenter $\alpha[f]$} (equivalently, after translating so that $\alpha[f]=0$). Let also $w\in[0,1]$ be the corresponding half–mass weight. Then, for the uncertainty product,
\begin{equation}\label{eq:cs-bound}
U[f]\ \ge\ \Big(w\,\sqrt{U[f_d]}+(1-w)\,\sqrt{U[f_s]}\Big)^2.
\end{equation}
In particular,
\begin{equation}\label{eq:min-bound}
U[f]\ \ge\ \min\{\,U[f_s],\,U[f_d]\,\}\qquad\text{for every }f,
\end{equation}
with equality in \eqref{eq:cs-bound} if and only if $\sigma_x^2[f_d] \sigma_\omega^2[\widehat{f_s}]= \sigma_x^2[f_s] \sigma_\omega^2[\widehat{f_d}]$ (Cauchy–Schwarz equality case).
\end{Thm}
We now state two theorems ensuring that the tent function used in \cite{mazzoccoli2024refining} is indeed a minimizer.
\begin{Thm}[The Tent Function Minimizes the Uncertainty Within the Family $\mathcal{G}_{\gamma,n}$]
\label{th:tent1}
Let $ \{\mathcal{G}_{\gamma,n}\} $ be the wavelet dictionary defined in equation~\eqref{definizioneunog}. Then, the following holds:
\begin{equation}
\inf_{g_n \in \{\mathcal{G}_{\gamma,n}\}_{n \in \mathbb{N}}} U[g_n] 
= \min_{g_n \in \{\mathcal{G}_{\gamma,n}\}_{n \in \mathbb{N}}} U[g_n] 
= U[g_1] = \frac{3}{10},
\label{infètriangolo1}
\end{equation}
where the minimum is attained by the function
\begin{equation}
\operatorname*{argmin}_{g_n \in \{\mathcal{G}_{\gamma,n}\}_{n \in \mathbb{N}}} U[g_n] 
= g_1(x) = \sqrt{\frac{3}{2t}} \left(1 - |x| \right) \, \mathds{1}_{[-1,1]}(x).
\end{equation}
Moreover,
\begin{equation}
\sup_{g_n \in \{\mathcal{G}_{\gamma,n}\}_{n \in \mathbb{N}}} U[g_n] = \frac{1}{2},
\label{supètriangolo1}
\end{equation}
which is asymptotically attained by the Dirac delta function.
\end{Thm}
\begin{Thm}[The Tent Function Minimizes the Uncertainty Within the Family $\mathcal{F}_{\gamma,n}$]
\label{th:tent2}
Let $ \{\mathcal{F}_{\gamma,n}\} $ be the wavelet dictionary defined in equation~\eqref{definizioneduef}.
Then, the uncertainty functional satisfies the following:
\begin{equation}
\inf_{f_n \in \{\mathcal{F}_{\gamma,n}\}_{n \in \mathbb{N}}} U[f_n] = \min_{f_n \in \{\mathcal{F}_{\gamma,n}\}_{n \in \mathbb{N}}} U[f_n] = U[f_1] = \frac{3}{10},
\label{infètriangolo1}
\end{equation}
where the minimum is attained by the function
\begin{equation}
\operatorname*{argmin}_{f_n \in \{\mathcal{F}_{\gamma,n}\}_{n \in \mathbb{N}}} U[f_n] = f_1(x) = \sqrt{\frac{3}{2t}} \left(1 - |x| \right) \, \mathds{1}_{[-1,1]}(x).
\end{equation}
Moreover, the uncertainty tends to grow linearly in the asymptotic regime:
\begin{equation}
\sup_{f_n \in \{\mathcal{F}_{\gamma,n}\}_{n \in \mathbb{N}}} U[f_n] \sim \frac{n}{6}\quad \text{as } n \to +\infty,
\label{supètriangolo2}
\end{equation}
and this supremum is effectively attained when the family is extended to include the rectangular function $ f_0(x) := \operatorname{rect}(x) $.
\end{Thm}
We now conclude by showing that the following theorem holds:
\begin{Thm}
\label{th:Urectp}
For all $p \ge 2$, there exist constants $c_1, c_2, c_3, c_4 \in \mathbb{R} \setminus \{0\}$, with $c_2 > c_1$ and $c_4 > c_3$, such that:
\begin{equation}
6\frac{\sqrt{p + c_2}}{(p + c_4)^{3/2}} \le \nu_p \le 6\frac{\sqrt{p + c_1}}{(p + c_3)^{3/2}} \label{disrapportofrequenze}
\end{equation}
and there exist constants $d_1, d_2, d_3, d_4 \in \mathbb{R} \setminus \{0\}$, with $d_2 > d_1$ and $d_4 > d_3$, such that:
\begin{equation}
\frac{\sqrt{(p + d_2)(p + d_3)}}{24} \le u_p \le \frac{\sqrt{(p + d_1)(p + d_4)}}{24} .\label{disrapportospaziale}
\end{equation}
As a result, the uncertainty of the function $\operatorname{rect}^{\{p\}}(x)$ satisfies:
\begin{equation}
\frac{1}{4} \cdot \frac{\sqrt{(p + d_2)(p + d_3)(p + c_2)}}{(p + c_4)^{3/2}} \le U(\operatorname{rect}^{{p}}) \le \frac{1}{4} \cdot \frac{\sqrt{(p + d_1)(p + d_4)(p + c_1)}}{(p + c_3)^{3/2}}. \label{disinceetezza}
\end{equation}
In particular, $\displaystyle\lim_{p \to +\infty} U(\operatorname{rect}^{\{p\}}) = \frac{1}{4}$, which is its infimum.
\label{teo2.6}
\end{Thm}}

%--- Section ---%
\section{Preliminary Results}\label{sec2}

We prove below that the integral defining the frequency mean $ \beta[\hat{f}] $ is always convergent; we will analyze its behavior near $ \omega = 0 $ and for large $ |\omega| $.

Since $ f $ has compact support, the Paley-Wiener theorem ensures that $ \hat{f} $ is smooth and square-integrable, i.e., $ \hat{f} \in L^2(\mathbb{R}) $. This means that $ |\hat{f}(\omega)|^2 $ is well-behaved and decays at infinity.

Since $ \hat{f} $ is smooth, we can expand it in a Taylor series around $ \omega = 0 $:  
$$
\hat{f}(\omega) = \hat{f}(0) + \omega \hat{f}'(0) + \frac{\omega^2}{2} \hat{f}''(0) + O(\omega^3).
$$
Squaring this expansion, we obtain:  
$$
|\hat{f}(\omega)|^2 = \hat{f}^2(0) + 2\omega \hat{f}(0) \hat{f}'(0) + O(\omega^2).
$$
Multiplying by $ \omega $ and integrating over $ (-1,1) $, we see that the integral  
$$
\int_{-1}^{1} \omega |\hat{f}(\omega)|^2 \, d\omega
$$
is convergent because the leading term is $ O(\omega) $, which integrates to $ O(1) $.

For large $ |\omega| $, smoothness of $ \hat{f} $ implies that it decays faster than any polynomial. In particular, integration by parts in the Fourier transform representation suggests a decay of the form  
$$
|\hat{f}(\omega)| \leq \frac{C}{|\omega|^N}
$$
for any $ N > 0 $, depending on the number of times we integrate by parts. Choosing $ N = 2 $, we get  
$$
|\hat{f}(\omega)|^2 \leq \frac{C}{|\omega|^4}.
$$
Thus, the integral  
$$
\int_{|\omega| \geq 1} \omega |\hat{f}(\omega)|^2 \, d\omega
$$
behaves asymptotically as  
$$
\int_{|\omega| \geq 1} \frac{\omega}{\omega^4} \, d\omega = \int_{|\omega| \geq 1} \frac{1}{\omega^3} \, d\omega,
$$
which is absolutely convergent.

Instead, the variance integral can diverge. Indeed, it is sufficient to consider the function
$$
f(x) = \frac{1}{2} \mathds{1}_{[-1,1]}(x),
$$  
whose Fourier transform is  
$$
\hat{f}(\omega) = \frac{\sin{\omega}}{\omega}.
$$  
From this, we obtain  
$$
\sigma_{\omega}^2 \left[\frac{\sin{\omega}}{\omega}\right] = +\infty.
$$  
It follows that the uncertainty can also diverge when taking the same function,  
$$
U\left[f\right] = +\infty.
$$

\subsection{Some useful Lemmas}

\begin{Lemma}[Uncertainty Properties]  
\label{lemma:uncertaintyprop}
Let $ K \subseteq \mathbb{R} $ be a compact set. Let $ U $ be the Uncertainty operator \eqref{incertezza}, then the following properties hold:  
\begin{itemize}  
    \item[(i)]  Let $ \lambda, \gamma \in \mathbb{R} \setminus \{0\}, \tau \in \mathbb{R} $, and define the function $ g_{\lambda,\gamma,\tau}(x) := \lambda f(\gamma x-\tau) $.  
Then, the following identity holds:
\begin{equation}
    U(g_{\lambda,\gamma,\tau}) = U(f), \quad \forall f \in \mathcal{F}_{\text{supp}}(K),
    \label{eq:supportonormalizzazione}
\end{equation}
where $ U(\cdot) $ denotes the uncertainty measure.
    \item[(ii)]  The following inequalities hold:
    \begin{equation}  
        \inf_{\substack{f \in \mathcal{F}_{supp}(K) \\ \left  \| f  \right\| \leq 1}}U[f] \leq \inf_{\substack{f \in \mathcal{F}_{supp}(K) \\ \left  \| f  \right\| = 1}}U[f] \; \;, \quad \quad\inf_{\substack{f \in \mathcal{F}_{supp}(K) \\ \left  \| f  \right\| \geq 1}}U[f] \leq \inf_{\substack{f \in \mathcal{F}_{supp}(K) \\ \left  \| f  \right\| = 1}}U[f]. 
\label{elencoincertezzesuF}  
    \end{equation}
    \item[(iii)] \textbf{[Non- Homogeneity]} $ U[\lambda f] = U[f] $ \; $ \forall \lambda \in \mathbb{R}, $ \; $ \forall f \in \mathcal{F}_{supp}(K) $.
\end{itemize}  
\end{Lemma}

%\textcolor{orange}{Il terzo punto del Lemma è utile nel caso in cui si considera un dizionario traslato, riscalato e normalizzato partendo da certe funzioni madri. Allora basta calcolare l'incertezza sulla funzione madre.
%}

\begin{proof}
We split the proof by points:

(i)  We observe that if $ \operatorname{supp}\{f(x)\} = [-a, a] $, then the support of the scaled function is given by:
\begin{equation}
    \operatorname{supp}\{g_{\lambda,\gamma}(x)\} = \left[ \frac{\tau}{\gamma}-\frac{a}{\gamma}, \frac{\tau}{\gamma}+\frac{a}{\gamma} \right].
    \label{varsupp}
\end{equation}
Furthermore, the $ L^2 $-norms of the functions are related by:
\begin{equation}
    \left\| g_{\lambda,\gamma,\tau} \right\|^2_{L^2\left( \left[ \frac{\tau-a}{\gamma}, \frac{\tau+a}{\gamma} \right] \right)} 
    = \int_{\frac{\tau-a}{\gamma}}^{\frac{\tau+a}{\gamma}} |g_{\lambda,\gamma,\tau}(x)|^2 \, dx 
    = \frac{\lambda^2}{\gamma} \int_{-a}^{a} |f(x)|^2 \, dx 
    = \frac{\lambda^2}{\gamma} \left\| f \right\|^2_{L^2([-a,a])}.
    \label{uguanorm}
\end{equation}
Using this identity, we compute the spatial mean of $ g_{\lambda,\gamma} $ as follows:
\begin{equation}
\begin{split}
    \alpha[g_{\lambda,\gamma,\tau}] 
    &= \frac{1}{\left\| g_{\lambda,\gamma,\tau} \right\|^2_{L^2\left( \left[ \frac{\tau-a}{\gamma}, \frac{\tau+a}{\gamma} \right] \right)}} \int_{\frac{\tau-a}{\gamma}}^{\frac{\tau+a}{\gamma}} x \, |g_{\lambda,\gamma}(x)|^2 \, dx \\
    &= \frac{1}{\left\| g_{\lambda,\gamma,\tau} \right\|^2} \int_{\frac{\tau-a}{\gamma}}^{\frac{\tau+a}{\gamma}} x \, |\lambda f(\gamma x-\tau)|^2 \, dx \\
    &= \frac{\gamma \lambda^2}{\lambda^2 \left\| f \right\|^2_{L^2([-a,a])}} \int_{-a}^{a} \frac{x+\tau}{\gamma} \, |f(x)|^2 \cdot \frac{dx}{\gamma} 
    = \frac{\alpha[f]}{\gamma}+\displaystyle\frac{\tau}{\gamma}.
\end{split}
\label{medialambagspaziale}
\end{equation}
Therefore, for the computation of the spatial variance, we obtain:
\begin{equation}
\begin{split}
\sigma_x^2[g_{\lambda,\gamma,\tau}] 
&= \frac{1}{\left\| g_{\lambda,\gamma,\tau} \right\|^2_{L^2\left( \left[ \frac{\tau-a}{\gamma}, \frac{\tau+a}{\gamma} \right] \right)}} \int_{\frac{\tau-a}{\gamma}}^{\frac{\tau+a}{\gamma}} (x - \alpha[g_{\lambda,\gamma,\tau}])^2 \, |g_{\lambda,\gamma,\tau}(x)|^2 \, dx \\
&= \frac{\gamma \lambda^2}{\lambda^2 \left\| f \right\|^2_{L^2([-a,a])}} \int_{\frac{\tau-a}{\gamma}}^{\frac{\tau+a}{\gamma}} \left(x - \frac{\alpha[f]}{\gamma}-\displaystyle\frac{\tau}{\gamma}\right)^2 |f(\gamma x-\tau)|^2 \, dx \\
&= \frac{\gamma}{\left\| f \right\|^2_{L^2([-a,a])}} \int_{-a}^{a} \left( \frac{x+\tau}{\gamma}-\frac{\alpha[f]}{\gamma}-\displaystyle\frac{\tau}{\gamma} \right)^2 |f(x)|^2 \cdot \frac{dx}{\gamma} \\
&= \frac{1}{\gamma^2 \left\| f \right\|^2_{L^2([-a,a])}} \int_{-a}^{a} (x - \alpha[f])^2 |f(x)|^2 \, dx 
= \frac{\sigma_x^2[f]}{\gamma^2}.
\end{split}
\label{varianzalambagspaziale}
\end{equation}
Next, we compute the Fourier transform of $ g_{\lambda,\gamma} $:
\begin{equation}
\begin{split}
\widehat{g_{\lambda,\gamma,\tau}}(\omega) 
&= \int_{\mathbb{R}} e^{-i \omega x} g_{\lambda,\gamma,\tau}(x) \, dx 
= \lambda \int_{\mathbb{R}} e^{-i \omega x} f(\gamma x-\tau) \, dx \\
&= \lambda \int_{\mathbb{R}} e^{-i \omega \frac{x+\tau}{\gamma}} f(x) \cdot \frac{dx}{\gamma} 
= \frac{\lambda}{\gamma}e^{-i\omega\frac{\tau}{\gamma}} \, \widehat{f} \left( \frac{\omega}{\gamma} \right).
\end{split}
\label{trasffouri}
\end{equation}
Hence, by applying Plancherel’s Theorem, we obtain the corresponding relation for the $ L^2 $-norms in the frequency domain:
\begin{equation}
\left\| \widehat{g_{\lambda,\gamma,\tau}} \right\|^2_{L^2(\mathbb{R})} 
= 2\pi \left\| g_{\lambda,\gamma,\tau} \right\|^2_{L^2\left( \left[-\frac{a}{\gamma}, \frac{a}{\gamma} \right] \right)} 
= 2\pi \cdot \frac{\lambda^2}{\gamma} \left\| f \right\|^2_{L^2([-a,a])} 
= \frac{\lambda^2}{\gamma} \left\| \widehat{f} \right\|^2_{L^2(\mathbb{R})}.
\label{ugunormfrequenza}
\end{equation}
We now compute the frequency mean of the function $ g_{\lambda,\gamma,\tau} $. Using the Fourier transform computed in~\eqref{trasffouri}, we obtain:
\begin{equation}
\begin{split}
\beta[\widehat{g_{\lambda,\gamma,\tau}}] 
&= \frac{1}{2\pi \left\| g_{\lambda,\gamma,\tau} \right\|^2_{L^2\left( \left[\frac{\tau-a}{\gamma}, \frac{\tau+a}{\gamma} \right] \right)}} \int_{\mathbb{R}} \omega \left| \widehat{g_{\lambda,\gamma,\tau}}(\omega) \right|^2 \, d\omega \\
&= \frac{1}{2\pi \cdot \frac{\lambda^2}{\gamma} \left\| f \right\|^2_{L^2([-a,a])}} \int_{\mathbb{R}} \omega \left| \frac{\lambda}{\gamma}e^{-i\omega\frac{\tau}{\gamma}} \hat{f}\left( \frac{\omega}{\gamma} \right) \right|^2 \, d\omega \\
&= \frac{\gamma}{2\pi \left\| f \right\|^2_{L^2([-a,a])}} \int_{\mathbb{R}} \omega \left| \hat{f}\left( \omega \right) \right|^2 \, d\omega = \gamma \, \beta[\widehat{f}].
\end{split}
\label{uguaglianza media frequenza}
\end{equation}
We now compute the frequency variance of $ g_{\lambda,\gamma} $:
\begin{equation}
\begin{split}
\sigma_\omega^2[\widehat{g_{\lambda,\gamma,\tau}}] 
&= \frac{1}{2\pi \left\| g_{\lambda,\gamma,\tau} \right\|^2_{L^2\left( \left[\frac{\tau-a}{\gamma}, \frac{\tau+a}{\gamma} \right] \right)}} \int_{\mathbb{R}} \left( \omega - \beta[\widehat{g_{\lambda,\gamma,\tau}}] \right)^2 \left| \widehat{g_{\lambda,\gamma,\tau}}(\omega) \right|^2 \, d\omega \\
&= \frac{1}{2\pi \cdot \frac{\lambda^2}{\gamma} \left\| f \right\|^2_{L^2([-a,a])}} \int_{\mathbb{R}} \left( \omega - \gamma \beta[\hat{f}] \right)^2 \left| \frac{\lambda}{\gamma} e^{-i\omega
\frac{\tau}{\gamma}}\hat{f}\left( \frac{\omega}{\gamma} \right) \right|^2 \, d\omega \\
&= \frac{\gamma^2}{2\pi \left\| f \right\|^2_{L^2([-a,a])}} \int_{\mathbb{R}} \left( \omega - \beta[\hat{f}] \right)^2 \left| \hat{f}(\omega) \right|^2 \, d\omega = \gamma^2 \sigma_\omega^2[\widehat{f}].
\end{split}
\label{varianzafrequenza}
\end{equation}
Therefore, the uncertainty product remains invariant under the scaling:
\begin{equation}
U(g_{\lambda,\gamma}) = \sigma_x^2[g_{\lambda,\gamma,\tau}] \cdot \sigma_\omega^2[\widehat{g_{\lambda,\gamma,\tau}}] = \frac{1}{\gamma^2} \sigma_x^2[f] \cdot \gamma^2 \sigma_\omega^2[\widehat{f}] = U(f).
\label{uguaglianza incertezze riscalatmnto}
\end{equation}

(ii) Since if $f \in \mathcal{F}_{supp}(K)$ with $\left\| f \right\| = 1$, then either $f \in \mathcal{F}_{supp}(K)$ with $\left\| f \right\| \leq 1$ or $f \in \mathcal{F}_{supp}(K)$ with $\left\| f \right\| \geq 1$, the following inequalities hold (by set inclusion):
$$
\inf_{\substack{f \in \mathcal{F}_{supp}(K) \\ \left\| f \right\| \leq 1}} U[f] \leq \inf_{\substack{f \in \mathcal{F}_{supp}(K) \\ \left\| f \right\| = 1}} U[f],
$$
$$
\inf_{\substack{f \in \mathcal{F}_{supp}(K) \\ \left\| f \right\| \geq 1}} U[f] \leq \inf_{\substack{f \in \mathcal{F}_{supp}(K) \\ \left\| f \right\| = 1}} U[f].
$$

Analogous for the second inequality.

Lastly, the point (iii) is an immediate consequence of (i).
\end{proof}
\begin{remark}
Lemma \ref{lemma:uncertaintyprop} shows that the uncertainty operator is invariant under scaling, translation, and modulation. As a consequence, when working with a waveform dictionary, the uncertainty of each atom can be computed directly from the mother function
\end{remark}

%\begin{example}Fixing $ a > 0 $, let $ f \in P_0^+([-a,a]) \subseteq \mathcal{F}_{supp}([-a,a]) $ be any function with $ \left\| f \right\| = 1 $. Defining the continuous sequence of functions $ f_t(x) := t f(x) $ for $ t \in (0,1] $, it holds that $ \forall t \in (0,1] $, we have $ f_t \in P_0^+([-a,a]) $ and $ \left\| f_t \right\| = t \left\| f \right\| \leq 1 $.  From property (ii) of Lemma 2.1, it follows that $ U[f_t] = U[tf] = t^4 U[f] $, and consequently, $ U[f_t] < U[f_s] $ for every $ t < s $ with $ t,s \in (0,1] $.  Furthermore,  
%$$\inf_{t \in (0,1]} U[f_t] = \lim_{t \to 0^+} U[f_t] = 0.$$\end{example}

From equation~\eqref{eq:supportonormalizzazione}, it follows that the uncertainty remains unchanged under normalization or rescaling of the function.

We ask whether, given a function $ f \in \mathcal{F}_0^+([-a,a]) $, we can find an even function, namely a function $ g \in P_0^+([-a,a]) $, such that $ U[f] \geq U[g] $.

%\textcolor{black}{Forse non serve che sia necessarimaente positiva tanto prendo il quadrato. Da ricontrollare se si può alleggerire qualche ipotesi sul teorema successivo.}

What is certainly obvious is that, since $ P_0^+([-a,a]) \subseteq \mathcal{F}_0^+([-a,a]) $, it follows that:  
\begin{equation}  
    \inf_{\substack{f \in \mathcal{F}_{0}^+([-a,a])}} U[f] \leq \inf_{\substack{f \in P_0^+([-a,a])}} U[f]  .
\end{equation}

We are therefore asking whether the opposite inequality holds, leading to equality --- specifically, that is, whether the infimum of the uncertainty is always attained by even functions.  

In the following theorem, we will construct an even function with lower uncertainty, although we will not be able to make any claims about the norm of the minimizing function.

Before stating and proving the theorem, it is important to make a key observation.  

If $ f \in \mathcal{F}_{0}^+([-a,a]) $, then its Fourier transform can be written as $ \hat{f}(\omega) = \Re(\hat{f}) + i\Im(\hat{f}) $, whereas if $ f \in P_{0}^+([-a,a]) $, then $ \hat{f}(\omega) = \Re(\hat{f}) $ is a real-valued function.  

This is because if $ f $ is even, then $ \hat{f} $ is also even. In fact,  
$$
\hat{f}(-\omega) = \int_{\mathbb{R}} e^{i\omega x} f(x) \, dx = \int_{\mathbb{R}} e^{-i\omega x} f(-x) \, dx,
$$  
which, by the evenness of $ f(x) $, is equal to $ \hat{f}(\omega) $.  

Thus,  
$$
\overline{\hat{f}(\omega)} = \int_{\mathbb{R}} \overline{e^{-i\omega x} f(x)} \, dx = \int_{\mathbb{R}} e^{i\omega x} f(x) \, dx = \hat{f}(-\omega) = \hat{f}(\omega),
$$  
which follows from the evenness of the Fourier transform, implying that $ \hat{f} $ is a real-valued function.

%However, we will always consider the squared modulus of the Fourier transform, which remains a real function in any case.

This observation is essential for the development of the subsequent results. In the following we denote by:
\textcolor{black}{
\begin{equation}
    \hat{f}_i(\omega) = \int_{\mathbb{R}} e^{-i\omega x} f_i(x) \, dx, \quad \forall i \in \{s, d\}
\end{equation}
}
the Fourier transforms of the left and right parts of the function, respectively.

\textcolor{black}{
We begin with an energy-balance statement that governs the sign of the cross-frequency term entering the main theorem of this section. By expressing the cross integral as the difference between the $L^2$ energies of the odd and even components of $f^{\prime}$, the proposition provides a sharp criterion for when the contribution is favorable or adverse.
\begin{Prop}[Cross–frequency identity and sign characterization]
\label{lemma:PV}
Let $a>0$ and $f\in \mathcal{F}_0^+([-a,a])$ be such that $f$ is absolutely continuous on $[-a,a]$ with weak derivative $f'\in L^2(-a,a)$. Let $f_d,f_s$ be the right/left reflections defined in \eqref{eq:simmetryfunctions}.
Then
\begin{equation}
\label{eq:cross-identity}
\int_{\mathbb{R}}\omega^2\,\hat f_s(\omega)\,\hat f_d(\omega)\,d\omega
=2\pi\int_{-a}^a f_s'(x)\,f_d'(x)\,dx
=-\,2\pi\int_{-a}^a f'(x)\,f'(-x)\,dx.
\end{equation}
Writing $u:=f'$ and decomposing $u$ into its even/odd parts
$$
u_{\mathrm{even}}(x):=\tfrac12\big(u(x)+u(-x)\big),\qquad
u_{\mathrm{odd}}(x):=\tfrac12\big(u(x)-u(-x)\big),
$$
we have the exact identity
\begin{equation}
\label{eq:even-odd-balance}
\int_{\mathbb{R}}\omega^2\,\hat f_s(\omega)\,\hat f_d(\omega)\,d\omega
=2\pi\Big(\,\|u_{\mathrm{odd}}\|_{L^2(-a,a)}^2-\|u_{\mathrm{even}}\|_{L^2(-a,a)}^2\,\Big).
\end{equation}
In particular, the sign of the cross term is \emph{not determined a priori}:
$$
\int_{\mathbb{R}}\omega^2\,\hat f_s\,\hat f_d\,d\omega\ \ge 0
\quad\Longleftrightarrow\quad
\|u_{\mathrm{odd}}\|_{2}\ \ge\ \|u_{\mathrm{even}}\|_{2}.
$$
\end{Prop}}
\begin{proof}
\textcolor{black}{
Since $f_s,f_d$ are real and even, their Fourier transforms $\hat f_s,\hat f_d$ are real and even. By Plancherel,
$$
\int_{\mathbb{R}}(i\omega \hat f_s)(\,\overline{i\omega \hat f_d}\,)\,d\omega
=2\pi\int_{-a}^a f_s'(x)\,f_d'(x)\,dx.
$$
Because $\hat f_s,\hat f_d$ are real, $\overline{i\omega \hat f_d}=-\,i\omega \hat f_d$, so the left-hand side equals $\int \omega^2 \hat f_s \hat f_d$. This yields the first equality in \eqref{eq:cross-identity}.
}

\textcolor{black}{
A direct check from \eqref{eq:simmetryfunctions} gives, for $x\neq 0$,
$$
f_d'(x)=\begin{cases} f'(x) & \text{if } x>0\\ -\,f'(-x) & \text{if } x<0\end{cases},
\qquad
f_s'(x)=\begin{cases} -\,f'(-x) &  \text{if } x>0\\ f'(x)& \text{if } x<0\end{cases},
$$
hence $f_s'(x)f_d'(x)=-\,f'(x)f'(-x)$ a.e. on $(-a,a)$, giving the second equality in \eqref{eq:cross-identity}.}

\textcolor{black}{Define $u:=f'$ and the reflection operator $R:L^2(-a,a)\to L^2(-a,a)$ by $(Ru)(x)=u(-x)$. Then
$$
\int_{-a}^a f'(x)f'(-x)\,dx=\langle u,Ru\rangle
=\|u_{\mathrm{even}}\|_2^2-\|u_{\mathrm{odd}}\|_2^2,
$$
because $Ru_{\mathrm{even}}=+u_{\mathrm{even}}$ and $Ru_{\mathrm{odd}}=-u_{\mathrm{odd}}$. Plugging this into \eqref{eq:cross-identity} yields \eqref{eq:even-odd-balance} and the sign characterization.}
\end{proof}

\textcolor{black}{
\begin{remark}[Even/unimodal case]
\label{rem:unimodal}
If $f$ is \emph{even} (for instance, even and unimodal, decreasing in $|x|$ from $0$), then $u=f'$ is \emph{odd} a.e., hence $u_{\mathrm{even}}\equiv 0$. Therefore
$$
\int_{\mathbb{R}}\omega^2\,\hat f_s(\omega)\,\hat f_d(\omega)\,d\omega
=2\pi\,\|f'\|_{L^2(-a,a)}^2\ \ge\ 0,
$$
with strict positivity whenever $f'\not\equiv 0$. Thus, under this structural assumption the cross term in frequency is nonnegative.
\end{remark}
\begin{example}[Sign can be positive or negative]\label{ex:sign-balance}
Consider the piecewise–cubic $f:[-1,1]\to\mathbb{R}$ defined by
$$
f(x)=
\begin{cases}
p_+(x) &\text{if }  0\le x\le 1\\
p_-(x) &\text{if } -1\le x<0
\end{cases}
\qquad
\begin{aligned}
p_+(x)&:=a_+x^3+b_+x^2+c_+x+d,\\
p_-(x)&:=a_-x^3+b_-x^2+c_-x+d,
\end{aligned}
$$
with coefficient tuples
\begin{align*}
&(a_+,b_+,c_+,d)=(0.0095109093,\,-1.6176420481,\,1.3050416889,\,0.3030894498),\\
&(a_-,b_-,c_-)=(0.5526139574,\,1.2275239864,\,0.9779994788).
\end{align*}
Numerically, $f(0^-)=f(0^+)=0.3030894498$, while $f(-1)\approx 5.6\!\times\!10^{-17}$ and $f(1)\approx -10^{-10}$ (roundoff).  
Let $u=f'$, and decompose $u=u_{\mathrm{even}}+u_{\mathrm{odd}}$. We obtain
$$
\|u_{\mathrm{even}}\|_2^2 \approx 0.675886085,\qquad
\|u_{\mathrm{odd}}\|_2^2  \approx 0.433013302,
$$
hence
$$
J:=\int_{-1}^1 f'(x)f'(-x)\,dx
= \|u_{\mathrm{even}}\|_2^2-\|u_{\mathrm{odd}}\|_2^2
\approx 0.242872783 \;>\; 0.
$$
By \eqref{eq:cross-identity}–\eqref{eq:even-odd-balance}, the cross term in frequency is therefore
$$
\int_{\mathbb{R}}\omega^2\,\hat f_s(\omega)\,\hat f_d(\omega)\,d\omega
= 2\pi\big(\|u_{\mathrm{odd}}\|_2^2-\|u_{\mathrm{even}}\|_2^2\big)
= -\,2\pi J
\approx -1.526014699 \;<\; 0.
$$
Figure~\ref{fig:piecewise-cubic-energy}\subref{fig:f-plot} shows $f$ and
Figure~\ref{fig:piecewise-cubic-energy}\subref{fig:u-parts} the even/odd parts of $u$.
\end{example}
}

% Place floats OUTSIDE theorem-like environments
\begin{figure}[hbt!]
  \centering

  \begin{subfigure}{.48\linewidth}
    \centering
    \includegraphics[width=\linewidth]{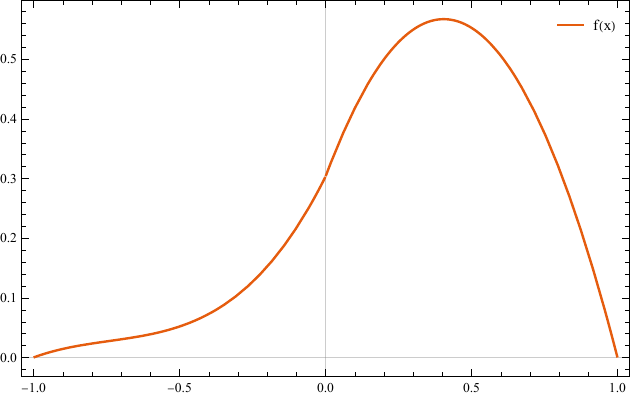}
    \caption{Test function $f$ on $[-1,1]$.}
    \label{fig:f-plot}
  \end{subfigure}\hfill
  \begin{subfigure}{.48\linewidth}
    \centering
    \includegraphics[width=\linewidth]{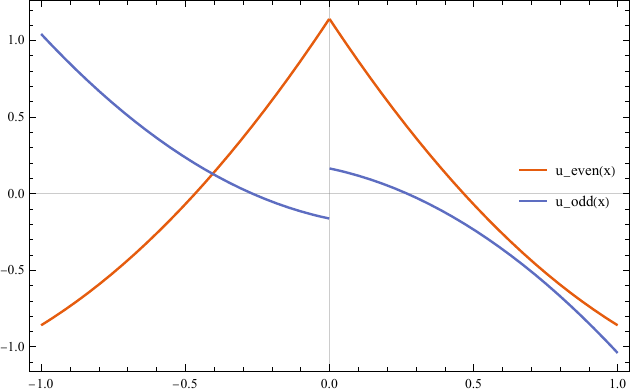}
    \caption{Even/odd parts of $u=f'$: $u_{\mathrm{even}},u_{\mathrm{odd}}$.}
    \label{fig:u-parts}
  \end{subfigure}
  \caption{Piecewise-cubic counterexample where $\int \omega^2 \hat f_s \hat f_d<0$.
  Numerically, $J\approx 0.242872783>0$ and $2\pi(\|u_{\mathrm{odd}}\|_2^2-\|u_{\mathrm{even}}\|_2^2)=-2\pi J\approx -1.526014699$.}
  \label{fig:piecewise-cubic-energy}
\end{figure}

We now turn our attention to the Fourier transform of the function. In order to establish the main theorem, we require a second technical lemma, stated below.

\begin{Lemma}
\label{lemma:LdL}
Under the assumptions of Lemma \ref{lemma:PV}, the following properties hold for the Fourier transforms:
\begin{itemize}
    \item[(i)] 
    $$
    \hat{f}(\omega) = \frac{1}{2} \left(  \hat{f}_s(\omega) + \hat{f}_d(\omega) \right) + i \, \Im(\hat{g}_d(\omega) - \hat{g}_s(\omega)),
    $$
    where the functions $\hat{f}_i(\omega)$ are even, and the imaginary parts $\Im(\hat{g}_i(\omega))$ are defined as
    $$
    \Im(\hat{g}_i(\omega)) := \frac{1}{2\pi \|f\|^2} \int_{0}^{a} \sin(\omega x) f_i(x) \, dx, \quad \text{for all } i \in \{s,d\},
    $$
    and are odd functions.
    \item[(ii)] The squared modulus $|\hat{f}(\omega)|^2$ is an even function, and satisfies the identity:
    \begin{equation}
    |\hat{f}(\omega)|^2 = \frac{1}{4} \left( \hat{f}_s(\omega) + \hat{f}_d(\omega) \right)^2 + \left( \Im(\hat{g}_d(\omega) - \hat{g}_s(\omega)) \right)^2
    \label{eqmoduloquadrofourier}
    \end{equation}
\end{itemize}
\end{Lemma}
\begin{proof}
Let us first compute the real and imaginary parts of $ \hat{f}(\omega) $. The real part is given by:
\begin{align*}
    \Re(\hat{f}(\omega)) = \int_{\mathbb{R}} \cos(\omega x) f(x) \, dx = \int_{-a}^{0} \cos(\omega x) f_s(x) \, dx + \int_{0}^{a} \cos(\omega x) f_d(x) \, dx.
\end{align*}
 Since both $ f_s $ and $ f_d $ are supported on symmetric intervals and are even functions, their Fourier transforms $ \hat{f}_s(\omega) $ and $ \hat{f}_d(\omega) $ are also even. Therefore, we obtain:
\begin{equation}
    \Re(\hat{f}(\omega)) = \frac{1}{2}\left( \hat{f}_s(\omega) + \hat{f}_d(\omega)\right).
\end{equation}
For the imaginary part, we compute:
\begin{equation}
\begin{split}
    \Im(\hat{f}(\omega)) &= \int_{\mathbb{R}} \sin(\omega x) f(x) \, dx \\
    &=  \int_{-a}^{0} \sin(\omega x) f_s(x) \, dx +\int_{0}^{a} \sin(\omega x) f_d(x) \, dx \\
    &= \Im(\hat{g}_d(\omega)) - \Im(\hat{g}_s(\omega)) = \Im(\hat{g}_d(\omega) - \hat{g}_s(\omega)),
\end{split}
\end{equation}
where we define:
$$
\Im(\hat{g}_i(\omega)) := \int_0^a \sin(\omega x) f_i(x) \, dx, \quad \forall i \in \{s, d\}.
$$
Note that each $ \Im(\hat{g}_i(\omega)) $ is an odd function of $ \omega $, since:
\begin{align*}
\Im(\hat{g}_i(-\omega)) &= \int_0^a \sin(-\omega x) f_i(x) \, dx \\
                        &= - \int_0^a \sin(\omega x) f_i(x) \, dx = -\Im(\hat{g}_i(\omega)).
\end{align*}
Thus, $ \Im(\hat{g}_d(\omega) - \hat{g}_s(\omega)) $ is also an odd function, being the difference of two odd functions.  
On the other hand, the quantity
$$
\frac{1}{2} \left(  \hat{f}_s(\omega) + \hat{f}_d(\omega) \right)
$$
is even, being a linear combination of even functions.

It immediately follows that the squared modulus $ |\hat{f}(\omega)|^2 $ is an even function, because it is the sum of the square of an even function and the square of an odd function:
\begin{equation}
\begin{split}
|\hat{f}(\omega)|^2 &= \left| \frac{1}{2} \left( \hat{f}_s(\omega) + \hat{f}_d(\omega) \right) + i \, \Im(\hat{g}_d(\omega) - \hat{g}_s(\omega)) \right|^2 \\
&= \frac{1}{4} \left(  \hat{f}_s(\omega) +  \hat{f}_d(\omega) \right)^2 + \left( \Im(\hat{g}_d(\omega) - \hat{g}_s(\omega)) \right)^2.
\end{split}
\end{equation}
This confirms the identity \eqref{eqmoduloquadrofourier}.  
\end{proof}

\textcolor{black}{\subsection{Proof of Theorem \ref{th:infimumevenf}}
The proof of Theorem \ref{th:infimumevenf} shows that the inequality}
$$
U[f]\ \ge\ \min\{\,U[f_s],\,U[f_d]\,\}\qquad\text{for every }f.
$$
Note that the choice of $f_{\min}$ must be consistent across both spatial and frequency domains, i.e., it must minimize the product $\sigma_x^2[f_i] \cdot \sigma_\omega^2[\hat{f}_i]$ jointly.

%So, given $ g(x) := c \cdot f_{\text{min}}(x) $, we have that $ g \in P_0^+([-a,a]) $ and minimizes the uncertainty. 

This leads to the following corollary:
\begin{Cor}
    We have that:
    $$
\inf_{\substack{f \in \mathcal{F}_0^+([-a,a])}} U[f] = \inf_{\substack{f \in P_0^+([-a,a])}} U[f]
 \textcolor{black}{ =\frac{1}{4}}.$$
\end{Cor}
\begin{proof}
    This follows directly from Theorem \ref{th:infimumevenf}.
\end{proof}

\textcolor{black}{
Invoking Lemma~\ref{lemma:LdL} we can write, for every $f\in\mathcal F_0^+([-a,a])$,
\begin{align*}
\sigma_\omega^2[\hat f]
&=\frac{1}{2\pi\|f\|_2^2}\int_{\mathbb R}\omega^2|\hat f(\omega)|^2\,d\omega\\
&=\frac{1}{4}\,\frac{1}{2\pi\|f\|_2^2}\int_{\mathbb R}\omega^2\Big(
\hat f_s(\omega)+\hat f_d(\omega)\Big)^2\,d\omega
\;+\;\frac{1}{2\pi\|f\|_2^2}\int_{\mathbb R}\omega^2\Big(\Im(\hat g_d-\hat g_s)\Big)^2\,d\omega\, .
\end{align*}
Expanding the square isolates a \emph{cross term}
$$
\int_{\mathbb R}\omega^2\,\hat f_s(\omega)\,\hat f_d(\omega)\,d\omega,
$$
whose sign is \emph{a priori undetermined}. Proposition~\ref{lemma:PV} gives the exact identity
$$
\int_{\mathbb R}\omega^2\,\hat f_s(\omega)\,\hat f_d(\omega)\,d\omega
=2\pi\Big(\|u_{\mathrm{odd}}\|_2^2-\|u_{\mathrm{even}}\|_2^2\Big),
\qquad u:=f',
$$
so the cross term is positive when the odd part of $f'$ dominates, negative when the even part
dominates, and vanishes only in the balanced case. In particular, additional symmetry assumptions
(e.g.\ $f$ even/unimodal, hence $u$ odd) force the cross term to be nonnegative.}

\textcolor{black}{This observation explains why a naive lower bound
$U[f]\ge \min\{U[f_s],U[f_d]\}$ may fail if one reflects about an arbitrary axis:
the uncontrolled cross term couples with the \emph{uncentered} spatial variance
$\sigma_x^2[f]=\tfrac{1}{\|f\|_2^2}\int x^2|f|^2-\alpha[f]^2$ and can drive $U[f]$ below both
$U[f_s]$ and $U[f_d]$. The remedy is canonical: reflect about the barycenter $\alpha[f]$
(or, equivalently, translate $f$ so that $\alpha[f]=0$ and then reflect). In that centered
setting, both the spatial and the spectral variances decompose as genuine convex combinations
of the variances of the two reflections, and the weighted Cauchy–Schwarz step in the
next yields the robust bound
$$
U[f]\ \ge\ \Big(w\,\sqrt{U[f_d]}+(1-w)\,\sqrt{U[f_s]}\Big)^2\ \ge\ \min\{U[f_s],U[f_d]\}.
$$
}

Now we prove Theorem \ref{th:infimumevenf}.

\begin{proof}
\textcolor{black}{\emph{Step 1: spatial side.} Assume $\alpha[f]=0$ (by Lemma~\ref{lemma:uncertaintyprop}(i) we can recenter $f$ without changing $U$).
Split $[-a,a]$ into the left/right halves and set
$$
N_+ := \int_{0}^{a}|f(x)|^2\,dx,\qquad N_- := \int_{0}^{a}|f(-x)|^2\,dx
\quad\Rightarrow\quad \|f\|_2^2=N_++N_- =\frac{1}{2}(\left\|f_s\right\|^2+\left\|f_d\right\|^2),$$
Write $f(x)=f_s(x)\,\mathds 1_{[-a,0)}(x)+f_d(x)\,\mathds 1_{[0,a]}(x)$, so with $\alpha[f]=0$,
$$
\sigma_x^2[f]=\frac{1}{\|f\|_2^2}\int_{-a}^a x^2|f(x)|^2\,dx
=\frac{1}{\|f\|_2^2}\!\left(\int_{-a}^0 x^2|f_s(x)|^2\,dx+\int_0^a x^2|f_d(x)|^2\,dx\right).
$$
Since $f_s,f_d$ are even,
$$
\int_{-a}^0 x^2|f_s|^2dx=\tfrac12\!\int_{-a}^a x^2|f_s|^2dx=\tfrac12\|f_s\|_2^2\,\sigma_x^2[f_s],\quad
\int_0^a x^2|f_d|^2dx=\tfrac12\|f_d\|_2^2\,\sigma_x^2[f_d].
$$
Now using $\|f_s\|_2^2=2N_{-}$ and $\|f_d\|_2^2=2N_{+}$ yelds
$$\sigma_x^2[f]=\frac{\int_{-a}^a x^2|f|^2dx}{\int_{-a}^a |f|^2dx}
= \frac{\|f_s\|_2^2\,\sigma_x^2[f_s] + \|f_d\|_2^2\,\sigma_x^2[f_d]}{\|f_s\|_2^2+\|f_d\|_2^2}
= w\,\sigma_x^2[f_d] + (1-w)\,\sigma_x^2[f_s],$$
where
\begin{equation}
\label{eq:w}
w:=\frac{N_+}{N_+ + N_-}
=\frac{\|f_d\|_2^2}{\|f_d\|_2^2+\|f_s\|_2^2}\in[0,1].    
\end{equation}
\emph{Step 2: frequency side.} By Lemma~\ref{lemma:LdL} we have $\beta[\hat f]=0$ (since $|\hat f|^2$ is even). Split derivative energy on the two halves:
$$
M_+ := \int_{0}^{a}|f'(x)|^2\,dx,\qquad
M_- := \int_{0}^{a}|f'(-x)|^2\,dx.
$$
By Plancherel
\begin{align}
\label{eq:sigw-f}
\sigma_\omega^2[\hat f]
&=\frac{1}{2\pi\|f\|_2^2}\int_{\mathbb R}\omega^2|\hat f(\omega)|^2\,d\omega \notag \\
&=\frac{1}{2\pi\|f\|_2^2}\int_{\mathbb{R}}\left(i \omega \hat{f}(\omega)\right) \cdot \overline{i \omega \hat{f}(\omega)} d \omega
=\frac{\|f'\|_2^2}{\|f\|_2^2}
=\frac{M_+ + M_-}{N_+ + N_-}.
\end{align}
Recall
$$
f_d'(x)=\begin{cases}f'(x),& x>0,\\ -\,f'(-x),& x<0,\end{cases}
\quad\Rightarrow\quad
\|f_d'\|_2^2=\int_{0}^{a}|f'(x)|^2\,dx+\int_{-a}^{0}|f'(-x)|^2\,dx=2M_+,
$$
and similarly
$
\|f_s'\|_2^2=2M_-.
$
Since $f_d,f_s$ are real and even, $\beta[f_s]=\beta[f_d]=0$ and using Plancherel on each even reflection:
\begin{equation}\label{eq:sigw-fd}
\sigma_\omega^2[\widehat{f_d}]
=\frac{1}{2\pi\|f_d\|_2^2}\int_{\mathbb R}\omega^2|\widehat{f_d}(\omega)|^2\,d\omega
=\frac{\|f_d'\|_2^2}{\|f_d\|_2^2}
=\frac{2M_+}{2N_+}
=\frac{M_+}{N_+},
\end{equation}
and, by the same calculation,
\begin{equation}\label{eq:sigw-fs}
\sigma_\omega^2[\widehat{f_s}]
=\frac{\|f_s'\|_2^2}{\|f_s\|_2^2}
=\frac{2M_-}{2N_-}
=\frac{M_-}{N_-}.
\end{equation}
Combining \eqref{eq:w}, \eqref{eq:sigw-f}–\eqref{eq:sigw-fs}, we obtain the exact convex combination
\begin{equation}\label{eq:sigw-weighted}
\sigma_\omega^2[\hat f]
=\frac{M_+ + M_-}{N_+ + N_-}
=w\,\frac{M_+}{N_+}+(1-w)\,\frac{M_-}{N_-}
=w\,\sigma_\omega^2[\widehat{f_d}]+(1-w)\,\sigma_\omega^2[\widehat{f_s}],
\end{equation}
with the same weight $w$ that appears in the spatial-variance mixture.}

\textcolor{black}{
\emph{Step 3: product and ordering.} Now apply the weighted Cauchy–Schwarz inequality to the two-term families
$a_1:=\sigma_x^2[f_d]$, $a_2:=\sigma_x^2[f_s]$, $b_1:=\sigma_\omega^2[\widehat{f_d}]$, $b_2:=\sigma_\omega^2[\widehat{f_s}]$ with weights $w$ and $1-w$:
$$
\big(w a_1+(1-w) a_2\big)\big(w b_1+(1-w) b_2\big)\ \ge\
\big(w\sqrt{a_1b_1}+(1-w)\sqrt{a_2b_2}\big)^2,
$$
where we used that $(\sqrt{a_1b_2}-\sqrt{a_2b_1})^2 \ge 0$ and then $a_1b_2+a_2b_1 \ge 2\sqrt{a_1b_1a_2b_2}$.}

\textcolor{black}{
With $a_i$ and $b_i$ as above, this is precisely \eqref{eq:cs-bound}. Since for any $x,y>0$ and $w\in[0,1]$ one has
$w\sqrt{x}+(1-w)\sqrt{y}\ge \min\{\sqrt{x},\sqrt{y}\}$, \eqref{eq:min-bound} follows. Equality in \eqref{eq:cs-bound} is the Cauchy–Schwarz equality condition, i.e. $a_1 b_2= a_2b_1$.}

\end{proof}
\textcolor{black}{
\begin{remark}
    If $N_+=0$ or $N_-=0$, then $w\in\{0,1\}$ and $f$ coincides (a.e.) with one even reflection; all identities above reduce to equalities with $U[f]=U[f_d]$ or $U[f]=U[f_s]$.
\end{remark}}
\textcolor{black}{
\begin{remark}[Regularity]
The hypothesis can be stated succinctly as $f\in H^1_0(-a,a)$, which is precisely what ensures $\sigma_\omega^2<\infty$ and justifies the Plancherel steps. If $f\notin H^1$, then $\sigma_\omega^2=+\infty$ and the bound holds only in a vacuous extended sense.
\end{remark}}

\begin{remark}[Why centering at the barycenter matters]
\textcolor{black}{Let $f\in H^1_0(-a,a)$ and denote by $f_s^{(0)},f_d^{(0)}$ the even reflections of $f$ about $x=0$.
Write the half–masses $N_\pm^{(0)}:=\int_{0}^{a}|f(\pm x)|^2\,dx$, the weights
$w^{(0)}:=N_+^{(0)}/(N_+^{(0)}+N_-^{(0)})$, and set
$A_{i}^{(0)}:=\sigma_x^2[f_{i}^{(0)}]$, $B_{i}^{(0)}:=\sigma_\omega^2[\widehat{f_{i}^{(0)}}]$ with $i\in\{s,d\}$.
A direct computation gives the \emph{uncentered} identities
$$
\sigma_x^2[f]=w^{(0)}A_d^{(0)}+(1-w^{(0)})A_s^{(0)}-\alpha[f]^2,
\qquad
\sigma_\omega^2[\hat f]=w^{(0)}B_d^{(0)}+(1-w^{(0)})B_s^{(0)}.
$$
The extra negative term $-\alpha[f]^2$ in the spatial variance can drive
$U[f]=\sigma_x^2[f]\sigma_\omega^2[\hat f]$ \emph{below} both $U[f_s^{(0)}]$ and $U[f_d^{(0)}]$,
so the lower bound $U[f]\ge\min\{U[f_s^{(0)}],U[f_d^{(0)}]\}$ may fail if one reflects about an
arbitrary axis.}

\textcolor{black}{The fix is canonical: either translate $f$ so that $\alpha[f]=0$ and then reflect about $0$,
or reflect $f$ about its barycenter $\alpha[f]$. If we define the barycentric reflections
$$
f_s^{(\alpha)}(x):=\begin{cases} f(x) & \text{if } x<\alpha[f]\\ f(2\alpha[f]-x)& \text{if } x\ge \alpha[f]\end{cases},
\qquad
f_d^{(\alpha)}(x):=\begin{cases} f(2\alpha[f]-x)& \text{if  } x<\alpha[f]\\ f(x)& \text{if } x\ge \alpha[f]\end{cases},
$$
and set the corresponding half–masses $N_\pm^{(\alpha)}$ and weight $w^{(\alpha)}:=N_+^{(\alpha)}/(N_+^{(\alpha)}+N_-^{(\alpha)})$, 
then one has the \emph{centered} convex decompositions
$$
\sigma_x^2[f]=w^{(\alpha)}A_d^{(\alpha)}+(1-w^{(\alpha)})A_s^{(\alpha)},
\qquad
\sigma_\omega^2[\hat f]=w^{(\alpha)}B_d^{(\alpha)}+(1-w^{(\alpha)})B_s^{(\alpha)},
$$
with $A_{i}^{(\alpha)}:=\sigma_x^2[f_{i}^{(\alpha)}]$ and
$B_{i}^{(\alpha)}:=\sigma_\omega^2[\widehat{f_{i}^{(\alpha)}}]$ for $i\in\{s,d\}$.
Weighted Cauchy–Schwarz then yields
$$
U[f]\ \ge\ \Big(w^{(\alpha)}\sqrt{U[f_d^{(\alpha)}]}+(1-w^{(\alpha)})\sqrt{U[f_s^{(\alpha)}]}\Big)^2
\ \ge\ \min\{U[f_s^{(\alpha)}],U[f_d^{(\alpha)}]\}.
$$
At this point we have to observe that $f_s^{(\alpha)}(x)$ and $f_d^{(\alpha)}(x)$ have support that is not symmetric with respect to the origin.}
\end{remark}

\textcolor{black}{\begin{Cor}
Let $a>0$ and $f \in \mathcal{F}_0^{+}([-a,a])$ be absolutely continuous with $f'\in L^2(-a,a)$. Then there exists $\psi\in\mathcal{P}_0^{+}([-a,a])$, obtained from $f$ by translation/reflection/rescaling, such that
$$
U[f]\ \ge\ U[\psi].
$$
\end{Cor}}

\begin{proof}
\textcolor{black}{Let $\alpha:=\alpha[f]$ be the barycenter. Consider the barycentric reflections $f_s^{(\alpha)},f_d^{(\alpha)}$ (as in Remark following Theorem~\ref{th:infimumevenf}): they are even with respect to $x=\alpha$, vanish at the endpoints of their supports, and satisfy
$$
U[f]\ \ge\ \min\{\,U[f_s^{(\alpha)}],\,U[f_d^{(\alpha)}]\,\}
$$
by Theorem~\ref{th:infimumevenf}.}

\textcolor{black}{Both $f_s^{(\alpha)}$ and $f_d^{(\alpha)}$ have support symmetric about $\alpha$ but not about the origin. Writing their half–lengths as
$$
r_s:=a+\alpha, \qquad r_d:=a-\alpha,
$$
we have $\operatorname{supp} f_s^{(\alpha)}=[\,\alpha-r_s,\alpha+r_s\,]=[-a,\,2\alpha+a]$ and $\operatorname{supp} f_d^{(\alpha)}=[\,\alpha-r_d,\alpha+r_d\,]=[\,2\alpha-a,\,a\,]$, and in both cases the boundary values are $0$.}

\textcolor{black}{Now define, for $i\in\{s,d\}$,
$$
\psi_i(x)\ :=\ f_i^{(\alpha)}\!\left(\frac{r_i}{a}\,x+\alpha\right),\qquad x\in\mathbb{R}.
$$
Then $\psi_i$ is even with respect to $0$, supported on $[-a,a]$, nonnegative, and vanishes at $\pm a$, hence $\psi_i\in\mathcal{P}_0^{+}([-a,a])$. By the translation/scale invariance of $U$ (Lemma~\ref{lemma:uncertaintyprop}(i)),
$$
U[\psi_i]\ =\ U\!\left[f_i^{(\alpha)}\right],\qquad i\in\{s,d\}.
$$
Choose $\psi=\psi_i$ with $U[\psi_i]=\min\{U[f_s^{(\alpha)}],U[f_d^{(\alpha)}]\}$. Then
$$
U[f]\ \ge\ \min\{U[f_s^{(\alpha)}],U[f_d^{(\alpha)}]\}\ =\ \min\{U[\psi_s],U[\psi_d]\}\ =\ U[\psi],
$$
as claimed. The degenerate case $r_s=0$ or $r_d=0$ (all mass on one side) is covered by the same construction.}
\end{proof}

\begin{example}[Uncentered failure vs centered restoration]
\textcolor{black}{Consider again the piecewise–cubic $f$ considered in Example \ref{ex:sign-balance}. A symbolic calculation (exact integrals and Plancherel) gives the barycenter
$$
\alpha[f]\approx 0.384209038102.
$$
If we reflect about $0$, the uncertainties are
$$
U[f]\approx 0.328205910036,\qquad
U\big[f_s^{(0)}\big]\approx 0.488135390966,\qquad
U\big[f_d^{(0)}\big]\approx 1.064558791510,
$$
so $U[f]<\min\{U[f_s^{(0)}],U[f_d^{(0)}]\}$ and the uncentered bound fails, exactly as
explained by the $-\alpha[f]^2$ term.}

\textcolor{black}{Reflecting instead about the barycenter, we obtain
$$
U\big[f_s^{(\alpha)}\big]\approx 0.233608189515,\qquad
U\big[f_d^{(\alpha)}\big]\approx 0.365009365360,
$$
hence
$$
U[f]\approx 0.328205910036\ \ge\ \min\{\,0.233608189515,\ 0.365009365360\,\}
=U\big[f_s^{(\alpha)}\big],
$$
which confirms the centered inequality.}
\end{example}

We now conclude the first part of the analysis. In the second part, we will introduce specific wavelet dictionaries and focus on particular families of functions, for which we will compute the uncertainty and identify minimizing elements.

We will see that the construction proposed in \cite{mazzoccoli2024refining} yields a minimizer within a specific family, see also \cite{rivero2023solution}. In contrast, for the family of convolution-type functions considered in the stability analysis of Gabor frames in \cite{de2018stability}, the infimum of the uncertainty turns out to be infinite.

\section{Computation of Uncertainty in Three Specific Cases}
\label{sec3}

We now focus on two broad families of functions and determine, in each case, the function that minimizes the uncertainty.

In particular, we will show that for the family considered in Section~\ref{sec:incertezzafamigliefunzioni}, the minimizing function coincides with the one presented in the article by \cite{mazzoccoli2024refining}.

On the other hand, for the family of convolution-type functions generated by the rectangular function $ \operatorname{rect}(x) := \mathds{1}_{[-\frac{1}{2},\frac{1}{2}]}(x) $, we show that the infimum of the uncertainty is not attained within the family, but only asymptotically, and equals $ \frac{1}{4} $. This value aligns with the lower bound given by the Heisenberg uncertainty principle (see \cite{busch2007heisenberg} for further discussion). 

The optimal uncertainty product
\begin{equation}
\sigma^2_x \sigma^2_{\omega} = \frac{1}{4}
\end{equation} 
is attained by the Morlet wavelet, defined as
\begin{equation}
\psi(t) = \pi^{-1/4} \, e^{i \omega_0 t} \, e^{-t^2/2},
\end{equation}
for any real $ \omega_0 \in \mathbb{R} $. This complex-valued wavelet achieves optimal joint time–frequency localization regardless of the modulation frequency $ \omega_0 $. See \cite{mallat1999wavelet} for further details.

\subsection{Analysis of the Uncertainty over a Specific Function Family}
\label{sec:incertezzafamigliefunzioni}

We now visualize the unmodulated functions $ g_n(x) $ and $ f_n(x) $ introduced in equations~\eqref{definizioneunog} and~\eqref{definizioneduef}, for various values of $ n $. These plots illustrate how the regularity and shape of each family evolves with the parameter $ n $, highlighting the increasing concentration and smoothness of the functions as $ n $ increases.
\begin{figure}[hbt!]
    \centering
    \begin{subfigure}[b]{0.48\textwidth}
        \centering
        \includegraphics[width=\textwidth]{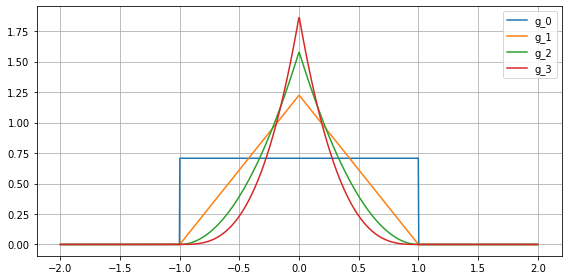}
        \caption{Profiles of the functions $ g_n(x) $ for $ n = 0, 1, 2, 3 $.}
        \label{fig:gn_family}
    \end{subfigure}
    \hfill
    \begin{subfigure}[b]{0.48\textwidth}
        \centering
        \includegraphics[width=\textwidth]{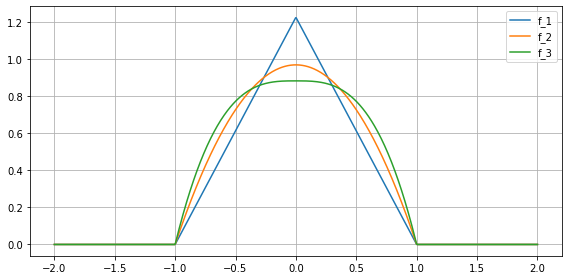}
        \caption{Profiles of the functions $ f_n(x) $ for $ n = 1, 2, 3$.}
        \label{fig:fn_family}
    \end{subfigure}
    \caption{Comparison of the unmodulated function families $ g_n(x) $ and $ f_n(x) $. As $ n $ increases, the functions become more localized and exhibit higher regularity near the origin.}
    \label{fig:gn_fn_comparison}
\end{figure}

\textcolor{black}{For spline-type atoms within affine dictionaries, the recent frame-set analysis provides an external benchmark to our minimization: it identifies parameter regimes where B-spline windows can or cannot generate Gabor frames (\cite{ghosh2025gabor,ghosh2025obstructions}). Related recent advances provide sharp uncertainty and local Lieb-type bounds directly in the wavelet setting, offering complementary constraints to our minimization within $\left\{G_{\gamma, n}\right\}$ and $\left\{F_{\gamma, n}\right\}$ (\cite{abreu2025donoho}). See also new Heisenberg-/Donoho-Stark-/logarithmic-type uncertainty principles derived for wavelet variants beyond the classical setting (\cite{dades2024heisenberg,dades2025new,wang2024benedicks}).}

\subsubsection{Some useful Lemmas}

To illustrate and compare the uncertainty properties of the two wavelet dictionaries introduced in Equations~\eqref{definizioneunog} and~\eqref{definizioneduef}, we compute the spatial and frequency means and variances of the associated atoms $ \mathcal{G}_{\gamma,n} $ and $ \mathcal{F}_{\gamma,n} $, respectively. These computations allow us to analyze how time–frequency localization depends on the structural form of the atoms. The results are summarized in Proposition \ref{prop:PV} and Proposition \ref{prop:PV2}.
\begin{Prop}[Uncertainty of $ \mathcal{G}_{\gamma,n}(x) $ dictionary]
\label{prop:PV}
Let the wavelet dictionary be defined as in equation~\eqref{definizioneunog}, and let $ \mathcal{G}_{\gamma,n}(x) $ denote an element of the set~\eqref{dizionario1}. Then, the spatial and frequency means of $ \mathcal{G}_{\gamma,n} $ are given by
$$
\alpha[\mathcal{G}_{\gamma,n}] = u, \quad \beta[\hat{\mathcal{G}}_{\gamma,n}] = 2\pi\xi.
$$
%\textcolor{orange}{serve però il $2\pi$ davanti}

Moreover, the spatial and frequency variances of $ \mathcal{G}_{\gamma,n} $ are, respectively,
$$
\sigma_x^2[\mathcal{G}_{\gamma,n}] = \frac{t^2}{2n^2 + 5n + 3},
$$
and
$$
\sigma_\omega^2[\hat{\mathcal{G}}_{\gamma,n}] = \frac{n^2 (2n + 1)}{t^2 (2n - 1)},
$$
for all integers $ n \in \mathbb{N} $, with $ n \ne 0 $.
\end{Prop}
\begin{proof}
    We recall that the spatial and frequency means are defined as:
$$
\alpha[f] := \frac{1}{\|f\|^2} \int_{\mathbb{R}} x |f(x)|^2 dx, \quad \beta[\hat{f}] := \frac{1}{2\pi \|f\|^2} \int_{\mathbb{R}} \omega |\hat{f}(\omega)|^2 d\omega.
$$
Since $ \|\mathcal{G}_{\gamma,n}\|=\|g_{n}\| = 1 $ (as previously shown), we can compute both means directly without normalizing.

We first compute:
$$
\alpha[\mathcal{G}_{\gamma,n}] = \int_{\mathbb{R}} x\, |\mathcal{G}_{\gamma,n}(x)|^2 \, dx.
$$
Note that:
$$
|\mathcal{G}_{\gamma,n}(x)|^2 = \frac{2n+1}{2t} \left(1 - \left| \frac{x - u}{t} \right| \right)^{2n} \mathds{1}_{[u - t, u + t]}(x),
$$
since the exponential factor $ e^{2\pi i \xi x} $ has unit modulus and thus does not affect the squared amplitude.

We perform the change of variable $ y = \frac{x - u}{t} $, so $ x = t y + u $ and $ dx = t\,dy $. The integral becomes:
$$
\alpha[\mathcal{G}_{\gamma,n}] = \frac{2n+1}{2t} \int_{u - t}^{u + t} x \left(1 - \left| \frac{x - u}{t} \right| \right)^{2n} dx
= \frac{2n+1}{2t} \cdot t \int_{-1}^{1} (ty + u)(1 - |y|)^{2n} dy.
$$
Distributing and simplifying:
$$
\alpha[\mathcal{G}_{\gamma,n}] = \frac{2n+1}{2} \int_{-1}^{1} (ty + u)(1 - |y|)^{2n} dy
= \frac{2n+1}{2} \left[ t \int_{-1}^{1} y(1 - |y|)^{2n} dy + u \int_{-1}^{1} (1 - |y|)^{2n} dy \right].
$$
The first integral is zero because the integrand is an odd function over a symmetric interval. The second integral evaluates to 1 due to the normalization of the function $ \mathcal{G}_{\gamma,n} $. Therefore,
$$
\alpha[\mathcal{G}_{\gamma,n}] = \frac{2n+1}{2} \cdot u \cdot \int_{-1}^{1} (1 - |y|)^{2n} dy = u.
$$
We now compute
$$
\beta[\hat{\mathcal{G}}_{\gamma,n}] = \frac{1}{2\pi} \int_{\mathbb{R}} \omega |\hat{\mathcal{G}}_{\gamma,n}(\omega)|^2 \, d\omega\, ,
$$
by analyzing the Fourier transform of the function
$$
\mathcal{G}_{\gamma,n}(x) = \sqrt{\frac{2n+1}{2t}} \left(1 - \left| \frac{x - u}{t} \right| \right)^n \mathds{1}_{[u - t, u + t]}(x) \, e^{2\pi i \xi x}.
$$
We begin by applying a change of variable. Let $ y := \frac{x - u}{t} $, so that $ x = t y + u $ and $ dx = t \, dy $. Note that the support of the function becomes $ y \in [-1,1] $. The Fourier transform is then given by:
$$
\begin{aligned}
\hat{\mathcal{G}}_{\gamma,n}(\omega)
&= \int_{\mathbb{R}} \mathcal{G}_{\gamma,n}(x) \, e^{-i \omega x} \, dx \\
&= \sqrt{\frac{2n+1}{2t}} \int_{u - t}^{u + t} \left(1 - \left| \frac{x - u}{t} \right| \right)^n e^{2\pi i \xi x} e^{-i \omega x} \, dx \\
&= \sqrt{\frac{2n+1}{2t}} \cdot t \cdot e^{i(2\pi \xi - \omega)u} \int_{-1}^{1} (1 - |y|)^n e^{i(2\pi \xi - \omega)t y} \, dy.
\end{aligned}
$$
Let us denote
$$
\tilde{\omega} := 2\pi \xi - \omega,
$$
so that
$$
\hat{\mathcal{G}}_{\gamma,n}(\omega) = \sqrt{\frac{2n+1}{2t}} \cdot t \cdot e^{i \tilde{\omega} u} \int_{-1}^{1} (1 - |y|)^n e^{i \tilde{\omega} t y} \, dy.
$$
Observe that both $ \tilde{\omega} t $ and $ y $ are real, and the integrand $ (1 - |y|)^n e^{i \tilde{\omega} t y} $ is the product of an even function and a complex exponential. Since $ (1 - |y|)^n $ is even and real-valued, and $ e^{i \tilde{\omega} t y} = \cos(\tilde{\omega} t y) + i \sin(\tilde{\omega} t y) $, the imaginary part integrates to zero due to the oddness of the sine function over a symmetric interval.

Hence, we conclude that
$$
\hat{\mathcal{G}}_{\gamma,n}(\omega) = 2\sqrt{\frac{2n+1}{2t}} \cdot t \cdot e^{i \tilde{\omega} u} \int_{0}^{1} (1 - y)^n \cos(\tilde{\omega} t y) \, dy.
$$
This provides an expression for the Fourier transform of $ \mathcal{G}_{\gamma,n} $ in terms of the cosine Fourier transform of the compactly supported, even function $ (1 - |y|)^n $. 

By introducing the function $ F_n $, defined as  
\begin{equation}
\label{eq:Fn}
F_n(\eta) := \int_0^1 (1 - y)^n \cos(\eta y) \, dy,
\end{equation}
which corresponds to the cosine Fourier transform of the real-valued, compactly supported function $ f_n(y) := (1 - y)^n \, \mathds{1}_{[0,1]}(y) $, the mean frequency $ \beta[\hat{\mathcal{G}}_{\gamma,n}] $ can be equivalently expressed in terms of $ F_n $ as follows:
$$
\beta[\hat{\mathcal{G}}_{\gamma,n}] = \frac{2n+1}{\pi} \cdot t  \int_{\mathbb{R}} \omega F_n^2(\omega t-2\pi \xi t) \, d\omega
$$
We now perform the change of variable $ \eta = \omega t - 2\pi \xi t $, which implies $ \omega = \frac{\eta}{t} + 2\pi \xi $ and $ d\omega = \frac{1}{t} d\eta $. Substituting into the integral, we obtain:
$$
\beta[\hat{\mathcal{G}}_{\gamma,n}] = \frac{2n+1}{\pi} \int_{\mathbb{R}} \left( \frac{\eta}{t} + 2\pi \xi \right) F_n^2(\eta)  \, d\eta.
$$
This simplifies to:
$$
\beta[\hat{\mathcal{G}}_{\gamma,n}] = \frac{2n+1}{\pi t} \int_{\mathbb{R}} \eta F_n^2(\eta) \, d\eta + \frac{2n+1}{\pi} \cdot 2\pi \xi \int_{\mathbb{R}} F_n^2(\eta) \, d\eta.
$$
Since $ F_n(\eta) $ is an even function, the first term vanishes because the integrand $ \eta F_n^2(\eta) $ is odd. Thus, we are left with:
$$
\beta[\hat{\mathcal{G}}_{\gamma,n}] = 2\xi(2n+1) \int_{\mathbb{R}} F_n^2(\eta) \, d\eta.
$$
Using the identity  
$$
\int_{\mathbb{R}} F_n^2(\eta) \, d\eta = \frac{\pi}{2n + 1},
$$  
which follows from Parseval's theorem applied to the cosine Fourier transform of the function $ f_n(y) := (1 - y)^n \, \mathds{1}_{[0,1]}(y) $, we obtain the final expression for the frequency mean:
$$
\beta[\hat{\mathcal{G}}_{\gamma,n}] = 2\pi \xi.
$$

We recall that the spatial variance is defined as:
$$
\sigma_x^2[f] := \frac{1}{\|f\|^2} \int_{\mathbb{R}} (x - \alpha[f])^2 |f(x)|^2 \, dx.
$$
Since $ \| \mathcal{G}_{\gamma,n} \| = 1 $ and $ \alpha[\mathcal{G}_{\gamma,n}] = u $ (as shown in the previous point), we can write:
$$
\sigma_x^2[\mathcal{G}_{\gamma,n}] = \int_{\mathbb{R}} (x - u)^2 |\mathcal{G}_{\gamma,n}(x)|^2 \, dx.
$$
We perform the change of variable:
$$
y := \frac{x - u}{t} \quad \Rightarrow \quad x = t y + u, \quad dx = t \, dy,
$$
which transforms the domain of integration to $ y \in [-1, 1] $. The variance becomes:
$$
\sigma_x^2[\mathcal{G}_{\gamma,n}] = \frac{2n+1}{2t} \int_{u - t}^{u + t} (x - u)^2 \left(1 - \left| \frac{x - u}{t} \right| \right)^{2n} dx
= \frac{2n+1}{2} t^2 \int_{-1}^{1} y^2 (1 - |y|)^{2n} dy.
$$
Since the integrand is even, we simplify:
$$
\int_{-1}^{1} y^2 (1 - |y|)^{2n} dy = 2 \int_0^1 y^2 (1 - y)^{2n} dy.
$$
Therefore:
$$
\sigma_x^2[\mathcal{G}_{\gamma,n}] = (2n+1)t^2 \cdot \int_0^1 y^2 (1 - y)^{2n} dy.
$$
This integral can be computed explicitly using the Beta function:
$$
\int_0^1 y^2 (1 - y)^{2n} dy = B(3, 2n + 1) = \frac{\Gamma(3)\Gamma(2n + 1)}{\Gamma(2n + 4)} = \frac{2! \cdot (2n)!}{(2n + 3)!}.
$$
Plugging this into the expression, we obtain:
$$
\sigma_x^2[\mathcal{G}_{\gamma,n}] = (2n+1)t^2 \cdot \frac{2(2n)!}{(2n+3)!}.
$$
Using the identity:
$$
\frac{(2n+1)(2n)!}{(2n+3)!} = \frac{1}{(2n+2)(2n+3)},
$$
we arrive at the simplified closed-form expression:
$$
\sigma_x^2[\mathcal{G}_{\gamma,n}] = \frac{2 t^2}{(2n + 2)(2n + 3)}=\frac{t^2}{2n^2+5n+3}.
$$

We now compute an explicit expression for the frequency variance $ \sigma_\omega^2[\hat{\mathcal{G}}_{\gamma,n}] $ associated with the wavelet atom $ \mathcal{G}_{\gamma,n} $. As previously shown, the frequency mean is $ \beta[\hat{\mathcal{G}}_{\gamma,n}] =2\pi \xi $, so by definition, the variance is given by:
$$
\sigma_\omega^2[\hat{\mathcal{G}}_{\gamma,n}] 
:= \frac{1}{2\pi} \int_{\mathbb{R}} (\omega - 2\pi \xi)^2 \cdot |\hat{\mathcal{G}}_{\gamma,n}(\omega)|^2 \, d\omega.
$$
As in the previous computation, we express the modulus squared of the Fourier transform in terms of the function $F_n$ defined by \eqref{eq:Fn}:
$$
\sigma_\omega^2[\hat{\mathcal{G}}_{\gamma,n}]
= (2n+1) \frac{t}{\pi} \int_{\mathbb{R}} (\omega - 2\pi \xi)^2 \cdot F_n^2\big( \omega t-2\pi \xi t \big) \, d\omega.
$$
Now apply the change of variable:
$$
\eta := \omega t-2\pi \xi t  \quad \Rightarrow \quad \omega = \frac{\eta}{t} + 2\pi \xi , \quad d\omega = \frac{1}{t} d\eta.
$$
Then:
$$
\sigma_\omega^2[\hat{\mathcal{G}}_{\gamma,n}]
= \frac{2n+1}{\pi t^2} \int_{\mathbb{R}} \eta ^2 \cdot F_n^2(\eta)  d\eta.
$$
We now invoke Parseval's Theorem. Let $ f(y) := (1 - y)^n \cdot \mathds{1}_{[0,1]}(y) $ and fix $n\neq 0$. Then %\footnote{\textcolor{black}{ mi sfugge la notazione $\mathcal{F}_{\cos}[f]$}}
$ F_n(\eta) = \mathcal{F}_{\cos}[f](\eta) $, the cosine transform of $ f $. Parseval’s identity for cosine transform gives:
$$
\int_{\mathbb{R}} \eta^2 F_n^2(\eta) \, d\eta = \pi \int_0^1 |f'(y)|^2 \, dy = \pi \int_0^1 n^2 (1 - y)^{2n - 2} dy = \pi \cdot \frac{n^2}{2n - 1}.
$$
Therefore:
$$
\sigma_\omega^2[\hat{\mathcal{G}}_{\gamma,n}]
=  \frac{2n+1}{\pi t^2} \cdot \pi \cdot \frac{n^2}{2n - 1}
=\frac{n^2 (2n + 1)}{t^2 (2n - 1)} .
$$
for $n\in \mathbb N$, $n\neq 0$.

\end{proof}

\begin{remark}
For $ n = 0 $, it is shown in \cite{mazzoccoli2024refining} that the frequency variance diverges:
\begin{equation}
\sigma_{\omega}^2[\hat{g}_0] = +\infty.
\label{formulavarianzafrequenzag0}
\end{equation}
For $ n = 1 $, the same reference establishes the explicit formula:
\begin{equation}
\sigma_{\omega}^2[\hat{g}_1] = \frac{3}{t^2}.
\label{formulavarianzafrequenzag1}
\end{equation}
\end{remark}

\textcolor{black}{
\begin{example}[Verification of Theorem \ref{th:tent1} for the dictionary $\{\mathcal{G}_{\gamma,n}\}_{\gamma\in\Gamma}$ at $n=1,2,3$.]
Fix $u\in\mathbb{R}$, $t>0$, and $\xi\in\mathbb{R}$. For $n=1,2,3$ consider
$$
h_{n,u,t}(x):=\sqrt{\frac{2n+1}{2t}}\,
\Bigl(1-\Bigl|\tfrac{x-u}{t}\Bigr|\Bigr)^{\!n}\,\mathds{1}_{[u-t,u+t]}(x)\; e^{2\pi i \xi x}.
$$
A direct change of variable $j=(x-u)/t$ shows that the three functions are normalized:
$$
\|h_{n,u,t}\|_{2}^{2}
=\frac{2n+1}{2t}\int_{u-t}^{u+t}\!\Bigl(1-\Bigl|\tfrac{x-u}{t}\Bigr|\Bigr)^{\!2n}\!dx
=\frac{2n+1}{2}\int_{-1}^{1}(1-|j|)^{2n}dj=1.
$$
Hence, by symmetry and modulation, the spatial and frequency means are
$$
\alpha[h_{n,u,t}]=u,\qquad \beta[\widehat{h}_{n,u,t}]=2\pi\xi.
$$
The spatial variances follow from the same substitution and the Beta function identity
$\int_{0}^{1} j^{2}(1-j)^{2n}dj=\frac{2}{(2n+1)(2n+2)(2n+3)}$, yielding
$$
\sigma_x^{2}[h_{n,u,t}]
=\frac{t^{2}}{\,2n^{2}+5n+3\,},
\quad\text{whence}\quad
\sigma_x^{2}[h_{1,u,t}]=\frac{t^{2}}{10},\;
\sigma_x^{2}[h_{2,u,t}]=\frac{t^{2}}{21},\;
\sigma_x^{2}[h_{3,u,t}]=\frac{t^{2}}{36}.
$$
For the Fourier transforms, write $s:=(2\pi\xi-\omega)t$. A direct computation gives
\begin{align*}
\widehat{h}_{1,u,t}(\omega)
&=e^{i(2\pi\xi-\omega)u}\sqrt{\frac{3t}{2}}\,
\Bigl(\frac{\sin(s/2)}{s/2}\Bigr)^{\!2},\\[2mm]
\widehat{h}_{2,u,t}(\omega)
&=e^{i(2\pi\xi-\omega)u}\sqrt{\frac{5t}{2}}\,
\frac{4\,\bigl(s-\sin s\bigr)}{s^{3}},\\[2mm]
\widehat{h}_{3,u,t}(\omega)
&=e^{i(2\pi\xi-\omega)u}\sqrt{\frac{7t}{2}}\,
\frac{12\bigl(\cos s-1+\tfrac{s^{2}}{2}\bigr)}{s^{4}}.
\end{align*}
By Plancherel, $\|\widehat{h}_{n,u,t}\|_{2}^{2}=2\pi$ for $n=1,2,3$. Using the identity
$$
\frac{1}{2\pi}\int_{\mathbb{R}}(\omega-2\pi\xi)^{2}\,|\widehat{h}_{n,u,t}(\omega)|^{2}\,d\omega
=\bigl\|\,\partial_x h_{n,u,t}\,\bigr\|_{2}^{2},
$$
one obtains the frequency variances
$$
\sigma_{\omega}^{2}[\widehat{h}_{n,u,t}]
=\frac{n^{2}(2n+1)}{(2n-1)\,t^{2}},
\quad\text{thus}\quad
\sigma_{\omega}^{2}[\widehat{h}_{1,u,t}]=\frac{3}{t^{2}},\;
\sigma_{\omega}^{2}[\widehat{h}_{2,u,t}]=\frac{20}{3t^{2}},\;
\sigma_{\omega}^{2}[\widehat{h}_{3,u,t}]=\frac{63}{5t^{2}}.
$$
Consequently, the uncertainty products
$$
U[h_{n,u,t}]:=\sigma_{x}^{2}[h_{n,u,t}]\,\sigma_{\omega}^{2}[\widehat{h}_{n,u,t}]
$$
agree with the closed-form expression in Proposition \ref{prop:PV} (and hence verify Theorem \ref{th:tent1}) and read
$$
U[h_{1,u,t}]=\frac{3}{10},\qquad
U[h_{2,u,t}]=\frac{20}{63},\qquad
U[h_{3,u,t}]=\frac{7}{20},
$$
which are increasing with $n$.
\end{example}}

\begin{Prop}[Uncertainty of $ \mathcal{F}_{\gamma,n} $ Dictionary]
\label{prop:PV2}
Let the wavelet dictionary be defined as in equation~\eqref{definizioneduef}, and let $ \mathcal{F}_{\gamma,n}(x) $ denote an element of the set~\eqref{dizionario2}. Then the spatial and frequency means of $ \mathcal{F}_{\gamma,n} $ are given by:
$$
\alpha[\mathcal{F}_{\gamma,n}] = u, \qquad \beta[\hat{\mathcal{F}}_{\gamma,n}] = 2\pi \xi.
$$
Moreover, the spatial and frequency variances of $ \mathcal{F}_{\gamma,n} $ are, respectively,
$$
\sigma_x^2[\mathcal{F}_{\gamma,n}] =  \frac{2 n^2+3n+1}{3 \left(2 n^2+9 n+9\right)}\, t^2 ,
$$
and
$$
\sigma_\omega^2[\hat{\mathcal{F}}_{\gamma,n}] = \frac{(2n+1)(n+1)}{2(2n-1)  t^2} ,
$$
for all integers $ n \in \mathbb{N} $, with $ n \ne 0 $.
\end{Prop}

\begin{proof}
We begin with the spatial mean:
$$
\alpha[\mathcal{F}_{\gamma,n}] = \int_{\mathbb{R}} x \cdot |\mathcal{F}_{\gamma,n}(x)|^2 \, dx.
$$
Since the complex exponential $ e^{2\pi i \xi x} $ has unit modulus, it does not affect the integral, so we can write:
\begin{equation}
\label{eq:F2}
|\mathcal{F}_{\gamma,n}(x)|^2 = \frac{(2n+1)(n+1)}{4n^2 t} \left(1 - \left| \frac{x - u}{t} \right|^n \right)^2 \mathds{1}_{[u - t, u + t]}(x).
\end{equation}
Performing the change of variable $ y = \frac{x - u}{t} $, so that $ x = ty + u $ and $ dx = t\,dy $, we obtain:
$$
\alpha[\mathcal{F}_{\gamma,n}] = \frac{(2n+1)(n+1)}{4n^2} \int_{-1}^{1} (ty + u) \left(1 - |y|^n \right)^2 \, dy.
$$
This expression splits into two terms:
$$
\alpha[\mathcal{F}_{\gamma,n}] = \frac{(2n+1)(n+1)}{4n^2} \left[ t \int_{-1}^{1} y \left(1 - |y|^n \right)^2 dy + u \int_{-1}^{1} \left(1 - |y|^n \right)^2 dy \right].
$$
The first integral is zero, since the integrand is an odd function over a symmetric interval. The second integral is equal to
$$
\int_{-1}^{1} \left(1 - |y|^n \right)^2 dy = \frac{4n^2}{(2n+1)(n+1)},
$$
which is chosen precisely to normalize the function in $ L^2 $. Hence, $\alpha[\mathcal{F}_{\gamma,n}] = u$.

We now compute the frequency mean:
$$
\beta[\hat{\mathcal{F}}_{\gamma,n}] = \frac{1}{2\pi} \int_{\mathbb{R}} \omega \cdot |\hat{\mathcal{F}}_{\gamma,n}(\omega)|^2 \, d\omega.
$$
We begin by computing the Fourier transform of the function $ \mathcal{F}_{\gamma,n}(x) $, defined as:
$$
\mathcal{F}_{\gamma,n}(x) = \sqrt{\frac{(2n+1)(n+1)}{4n^2 t}} \left(1 - \left| \frac{x - u}{t} \right|^n \right) \mathds{1}_{[u - t, u + t]}(x) \cdot e^{2\pi i \xi x}.
$$
We perform the change of variable $ y = \frac{x - u}{t} $, which leads to:
$$
\hat{\mathcal{F}}_{\gamma,n}(\omega) = \sqrt{\frac{(2n+1)(n+1)}{4n^2 t}} \cdot t \cdot e^{i(2\pi \xi - \omega) u} \int_{-1}^{1} (1 - |y|^n) e^{i(2\pi \xi - \omega)t y} \, dy.
$$
Since the function $ (1 - |y|^n) $ is real and even, and the exponential can be expressed in terms of cosine and sine, the integral reduces to a cosine transform:
\begin{equation*}
\hat{\mathcal{F}}_{\gamma,n}(\omega) = 2t \cdot \sqrt{\frac{(2n+1)(n+1)}{4n^2 t}} \cdot e^{i(2\pi \xi - \omega) u} \int_{0}^{1} (1 - y^n) \cos\big((2\pi \xi - \omega)t y\big) \, dy.
\end{equation*}
We define the auxiliary function:
$$
H_n(\eta) := \int_0^1 (1 - y^n) \cos(\eta y) \, dy,
$$
so that the Fourier transform becomes:
\begin{equation}
\label{eq:Fnfourier}
\hat{\mathcal{F}}_{\gamma,n}(\omega) = \sqrt{t} \cdot \sqrt{\frac{(2n+1)(n+1)}{n^2}} \cdot e^{i(2\pi \xi - \omega) u} \cdot H_n\big((\omega-2\pi \xi)t\big).    
\end{equation}
Hence, its squared modulus, substituted into the definition of $ \beta[\hat{\mathcal{F}}_{\gamma,n}] $, yields:
$$
\beta[\hat{\mathcal{F}}_{\gamma,n}] = \frac{t(2n+1)(n+1)}{2\pi n^2} \int_{\mathbb{R}} \omega \cdot H_n^2\big((\omega-2\pi \xi)t\big) \, d\omega.
$$
We now perform the change of variable:
$$
\eta := (\omega-2\pi \xi)t \quad \Rightarrow \quad \omega = 2\pi \xi + \frac{\eta}{t}, \quad d\omega = \frac{1}{t} d\eta.
$$
Thus, the integral becomes:
$$
\begin{aligned}
\beta[\hat{\mathcal{F}}_{\gamma,n}]
&= \frac{t(2n+1)(n+1)}{2\pi n^2} \cdot \int_{\mathbb{R}} \left(2\pi \xi + \frac{\eta}{t} \right) H_n^2(\eta) \cdot \left(\frac{1}{t} \right) \, d\eta \\
&= \frac{(2n+1)(n+1)}{2\pi n^2} \cdot \int_{\mathbb{R}} \left( \frac{\eta}{t} +2\pi \xi \right) H_n^2(\eta) \, d\eta.
\end{aligned}
$$
We now split the integral:
$$
\beta[\hat{\mathcal{F}}_{\gamma,n}] = \frac{(2n+1)(n+1)}{2\pi n^2} \left[ \frac{1}{t} \int_{\mathbb{R}} \eta \cdot H_n^2(\eta) \, d\eta + 2\pi \xi \int_{\mathbb{R}} H_n^2(\eta) \, d\eta \right].
$$
Since $ H_n(\eta) $ is an even function, the first integral vanishes. Therefore, we are left with:
$$
\beta[\hat{\mathcal{F}}_{\gamma,n}] = \frac{(2n+1)(n+1)}{2\pi n^2} \cdot 2\pi \xi \cdot \int_{\mathbb{R}} H_n^2(\eta) \, d\eta.
$$
Using Parseval's identity, applied to the cosine transform of the function $ y \mapsto 1 - y^n $ on $[0,1]$, we have:
$$
\int_{\mathbb{R}} H_n^2(\eta) \, d\eta = \pi \int_0^1 (1 - y^n)^2 \, dy.
$$
This integral simplifies to:
$$
\int_0^1 (1 - y^n)^2 \, dy = \int_0^1 \left( 1 - 2y^n + y^{2n} \right) dy = 1 - \frac{2}{n+1} + \frac{1}{2n+1}.
$$
Therefore, the full expression becomes $\beta[\hat{\mathcal{F}}_{\gamma,n}] = 2\pi \xi$.

We now compute the spatial variance $ \sigma_x^2[\mathcal{F}_{\gamma,n}] $, which is defined as:
$$
\sigma_x^2[\mathcal{F}_{\gamma,n}] := \int_{\mathbb{R}} (x - \alpha[\mathcal{F}_{\gamma,n}])^2 \cdot |\mathcal{F}_{\gamma,n}(x)|^2 \, dx.
$$
From the previous calculation, we already know that $ \alpha[\mathcal{F}_{\gamma,n}] = u $, so the variance simplifies to:
$$
\sigma_x^2[\mathcal{F}_{\gamma,n}] = \int_{\mathbb{R}} (x - u)^2 \cdot |\mathcal{F}_{\gamma,n}(x)|^2 \, dx.
$$
Recalling the \eqref{eq:F2}, we apply the change of variable $ y := \frac{x - u}{t} $, so that $ x = t y + u $, $ dx = t \, dy $, and the support becomes $ y \in [-1,1] $. This yields:
$$
\begin{aligned}
\sigma_x^2[\mathcal{F}_{\gamma,n}] 
&= \frac{(2n+1)(n+1)}{4n^2 t} \int_{u - t}^{u + t} (x - u)^2 \left(1 - \left| \frac{x - u}{t} \right|^n \right)^2 \, dx \\
&= \frac{(2n+1)(n+1)}{4n^2 t} \cdot t^3 \int_{-1}^{1} y^2 (1 - |y|^n)^2 \, dy.
\end{aligned}
$$
Since the integrand is even, we restrict to the interval $ [0,1] $ and double the integral:
$$
\int_{-1}^{1} y^2 (1 - |y|^n)^2 \, dy = 2 \int_0^1 y^2 (1 - y^n)^2 \, dy.
$$
We expand the squared term:
$$
(1 - y^n)^2 = 1 - 2y^n + y^{2n},
$$
and compute the integral:
$$
\int_0^1 y^2 (1 - y^n)^2 \, dy = \int_0^1 \left( y^2 - 2y^{n+2} + y^{2n+2} \right) \, dy = \frac{1}{3} - \frac{2}{n+3} + \frac{1}{2n+3}.
$$
Therefore, the final expression for the spatial variance becomes:
$$
\sigma_x^2[\mathcal{F}_{\gamma,n}] = \frac{2 n^2+3n+1}{3 \left(2 n^2+9 n+9\right)}\, t^2 .
$$

We now compute the frequency variance $ \sigma_\omega^2[\hat{\mathcal{F}}_{\gamma,n}] $, defined as
$$
\sigma_\omega^2[\hat{\mathcal{F}}_{\gamma,n}] := \frac{1}{2\pi} \int_{\mathbb{R}} (\omega - \beta[\hat{\mathcal{F}}_{\gamma,n}])^2 \cdot |\hat{\mathcal{F}}_{\gamma,n}(\omega)|^2 \, d\omega.
$$
From the previous calculation, we know that $ \beta[\hat{\mathcal{F}}_{\gamma,n}] = 2\pi \xi $. Therefore, we can rewrite the integral as:
$$
\sigma_\omega^2[\hat{\mathcal{F}}_{\gamma,n}] = \frac{1}{2\pi} \int_{\mathbb{R}} (\omega - 2\pi \xi)^2 \cdot |\hat{\mathcal{F}}_{\gamma,n}(\omega)|^2 \, d\omega.
$$
From \eqref{eq:Fnfourier}, we know:
$$
|\hat{\mathcal{F}}_{\gamma,n}(\omega)|^2 = \frac{(2n+1)(n+1)}{n^2}\, t \cdot H_n^2(t(\omega-2\pi \xi)),
$$
Performing the change of variable $ \eta = t(\omega-2\pi \xi) $, so that $ \omega = 2\pi \xi + \frac{\eta}{t} $ and $ d\omega = \frac{1}{t} \, d\eta $, we obtain:
$$
\sigma_\omega^2[\hat{\mathcal{F}}_{\gamma,n}]= \frac{(2n+1)(n+1)}{2\pi n^2 t^2} \cdot \int_{\mathbb{R}} \eta^2 H_n^2(\eta) \, d\eta.
$$
We now apply Parseval’s identity for the cosine transform. Let
$$
f_n(y) := (1 - y^n) \cdot \mathds{1}_{[0,1]}(y),
$$
so that $ H_n(\eta) $ is the cosine transform of $ f_n(y) $. Then Parseval's identity gives:
$$
\int_{\mathbb{R}} \eta^2 H_n^2(\eta) \, d\eta = \pi \int_0^1 |f_n'(y)|^2 \, dy.
$$
We compute:
$$
f_n'(y) = -n y^{n-1}, \quad \Rightarrow \quad |f_n'(y)|^2 = n^2 y^{2n-2}.
$$
Hence,
$$
\int_{\mathbb{R}} \eta^2 H_n^2(\eta) \, d\eta = \pi n^2 \int_0^1 y^{2n-2} \, dy = \pi n^2 \cdot \frac{1}{2n - 1}.
$$
Substituting back into the expression for the variance, we conclude:
$$
\sigma_\omega^2[\hat{\mathcal{F}}_{\gamma,n}] = \frac{(2n+1)(n+1)}{2(2n-1)  t^2}.
$$

\end{proof}

\textcolor{black}{\begin{example}[Explicit check of Theorem \ref{th:tent2} for the dictionary $\{\mathcal{F}_{\gamma,n}\}_{\gamma\in\Gamma}$ at $n=1,2,3$.]
Fix $u\in\mathbb{R}$, $t>0$, and $\xi\in\mathbb{R}$. Consider
\begin{align*}
h_{1,u,t}(x)&:=\sqrt{\frac{3}{2t}}\Bigl(1-\Bigl|\tfrac{x-u}{t}\Bigr|\Bigr)\,\mathds{1}_{[u-t,u+t]}(x)\,e^{2\pi i\xi x},\\
h_{2,u,t}(x)&:=\sqrt{\frac{15}{16t}}\Bigl(1-\Bigl|\tfrac{x-u}{t}\Bigr|^{2}\Bigr)\,\mathds{1}_{[u-t,u+t]}(x)\,e^{2\pi i\xi x},\\
h_{3,u,t}(x)&:=\sqrt{\frac{7}{9t}}\Bigl(1-\Bigl|\tfrac{x-u}{t}\Bigr|^{3}\Bigr)\,\mathds{1}_{[u-t,u+t]}(x)\,e^{2\pi i\xi x}.
\end{align*}
With the change of variable $j=(x-u)/t$, one checks that $\|h_{n,u,t}\|_{2}=1$ for $n=1,2,3$:
$$
\|h_{n,u,t}\|_{2}^{2}=c_n\int_{-1}^{1}\!\bigl(1-|j|^{n}\bigr)^{2}\,dj=1
\quad\bigl(c_1=\tfrac{3}{2},~c_2=\tfrac{15}{16},~c_3=\tfrac{7}{9}\bigr).
$$
By symmetry and modulation, the spatial and frequency means are
$$
\alpha[h_{n,u,t}]=u,\qquad \beta[\widehat{h}_{n,u,t}]=2\pi\xi.
$$
For the spatial variances, the same substitution yields
$$
\sigma_x^{2}[h_{n,u,t}]
=\frac{1}{\|h_{n,u,t}\|_{2}^{2}}\int_{u-t}^{u+t}(x-u)^{2}|h_{n,u,t}(x)|^{2}dx
= t^{2}\,\frac{2n^{2}+3n+1}{3\,(2n^{2}+9n+9)}.
$$
In particular,
$$
\sigma_x^{2}[h_{1,u,t}]=\frac{t^{2}}{10},\qquad
\sigma_x^{2}[h_{2,u,t}]=\frac{t^{2}}{7},\qquad
\sigma_x^{2}[h_{3,u,t}]=\frac{14\,t^{2}}{81}.
$$
Let $s:=(2\pi\xi-\omega)t$. A direct computation gives the Fourier transforms
\begin{align*}
\widehat{h}_{1,u,t}(\omega)
&=e^{i(2\pi\xi-\omega)u}\sqrt{\frac{3t}{2}}\left(\frac{\sin(s/2)}{s/2}\right)^{2},\\[2mm]
\widehat{h}_{2,u,t}(\omega)
&=e^{i(2\pi\xi-\omega)u}\sqrt{\frac{15t}{16}}\,
\frac{4\,\bigl(\sin s - s\cos s\bigr)}{s^{3}},\\[2mm]
\widehat{h}_{3,u,t}(\omega)
&=e^{i(2\pi\xi-\omega)u}\sqrt{\frac{7t}{9}}\,
\frac{12\,s\sin s+(12-6s^{2})\cos s -12}{s^{4}}.
\end{align*}
By Plancherel, $\|\widehat{h}_{n,u,t}\|_{2}^{2}=2\pi$ for $n=1,2,3$. Using the identity
$$
\frac{1}{2\pi}\int_{\mathbb{R}}(\omega-2\pi\xi)^{2}\,|\widehat{h}_{n,u,t}(\omega)|^{2}\,d\omega
=\bigl\|\,\partial_x h_{n,u,t}\,\bigr\|_{2}^{2},
$$
we obtain the frequency variances
$$
\sigma_{\omega}^{2}[\widehat{h}_{n,u,t}]
=\frac{(2n+1)(n+1)}{2(2n-1)\,t^{2}},
$$
hence
$$
\sigma_{\omega}^{2}[\widehat{h}_{1,u,t}]=\frac{3}{t^{2}},\qquad
\sigma_{\omega}^{2}[\widehat{h}_{2,u,t}]=\frac{5}{2t^{2}},\qquad
\sigma_{\omega}^{2}[\widehat{h}_{3,u,t}]=\frac{14}{5t^{2}}.
$$
Consequently, the uncertainty products
$$
U[h_{n,u,t}]:=\sigma_x^{2}[h_{n,u,t}]\,\sigma_{\omega}^{2}[\widehat{h}_{n,u,t}]
$$
agree with the closed form in Proposition \ref{prop:PV2} (thus verifying Theorem \ref{th:tent2}). For $n=1,2,3$:
$$
U[h_{1,u,t}]=\frac{3}{10},\qquad
U[h_{2,u,t}]=\frac{5}{14},\qquad
U[h_{3,u,t}]=\frac{196}{405},
$$
which increase with $n$.
\end{example}}

\subsubsection{Proofs of Theorems \ref{th:tent1} and \ref{th:tent2}}
\begin{proof}
The proof of Theorem \ref{th:tent1} is based on the explicit computation of the uncertainty, which can be derived from the spatial and frequency variances of $ \mathcal{G}_{\gamma,n} $ stated in Proposition~\ref{prop:PV}, i.e.,
$$U[\mathcal{G}_{\gamma,n}]= \sigma_x^2[\mathcal{G}_{\gamma,n}] \sigma_\omega^2[\hat{\mathcal{G}}_{\gamma,n}] = \frac{n^2 (2n + 1)}{(2n^2 + 5n + 3) (2n - 1)}  ,
$$
Therefore, the sequence $ U[g_n] $ is increasing for $ n \geq 1 $ and converges in the limit as $ n \to +\infty $ to the value $ \frac{1}{2} $. It follows that the minimum of the uncertainty is attained at $ n = 1 $, with value $ \frac{3}{10} $, while the supremum is reached asymptotically and equals $ \frac{1}{2} $.
\end{proof}

\begin{proof}
The proof of Theorem \ref{th:tent2} follows by direct substitution of the expressions for spatial and frequency variances derived in Proposition~\ref{prop:PV2}. Specifically,
\begin{equation} 
U(\mathcal{F}_{\gamma,n})=\sigma_x^2[\mathcal{F}_{\gamma,n}]\sigma_\omega^2[\hat{\mathcal{F}}_{\gamma,n}]= \frac{(2n+1)^2(n+1)^2}{6(2n-1) \left(2 n^2+9 n+9\right)}
\end{equation}
The proof follows from similar arguments to those used in the previous theorem, relying on explicit computations of the spatial and frequency variances as well as the structure of the uncertainty functional.
\end{proof}

\subsection{Analysis of Rectangular Function Convolutions}
\textcolor{black}{Let us recall that $ \operatorname{rect}^{\{p\}}(x) $ denotes the $ p $-th convolution of the rectangular function with itself.}

We now proceed by considering the case of variable support. To this end, we first introduce the appropriate wavelet dictionary. Then, in the spirit of the previous section, we carry out a similar analysis by defining a rescaled wavelet dictionary based on convolutions supported on the fixed interval $ [-1, 1] $. This allows us to focus purely on the shape of the optimal functions, independently of their scale. \textcolor{black}{Comparable families of atoms in generalized affine/metaplectic environments satisfy parallel uncertainty inequalities, such as the Heisenberg, logarithmic, and Pitt/Nazarov types, supporting the robustness of our asymptotic 1/4 limit under broader transforms (\cite{dar2024n,dades2025new}).}

What we will show is that the two problems are equivalent: the uncertainty values obtained from the two dictionaries coincide for every $ p \ge 1 $.

We begin by analyzing the variable-support case and introduce the corresponding dictionary:
\begin{equation}
\{\mathcal{C}_{\gamma,p}\}_{\gamma \in \Gamma} := \{\phi_{\gamma,p}(x)\}_{\gamma \in \Gamma} := \left\{ \sqrt{\frac{N_p}{t}} \, \operatorname{rect}^{\{p\}}\left( \frac{x - u}{t} \right) e^{2\pi i \xi x} \;\middle|\; p \ge 1 \right\}_{\gamma \in \Gamma},
\label{dizionarioconvoluzionesupvar}
\end{equation}
which may also be written in compact support form as:
$$
\left\{ \sqrt{\frac{N_p}{t}} \, \operatorname{rect}^{\{p\}}\left( \frac{x - u}{t} \right) \mathds{1}_{[u - \frac{pt}{2}, \, u + \frac{pt}{2}]}(x) \, e^{2\pi i \xi x} \;\middle|\; p \ge 1 \right\}_{\gamma \in \Gamma}.
$$
The normalization constant $ N_p $ in \eqref{dizionarioconvoluzionesupvar} is chosen so that each function in the dictionary has unit $ L^2 $-norm. Since $ \operatorname{rect}^{\{p\}} $ is compactly supported on the interval $ \left[-\frac{p}{2}, \frac{p}{2}\right] $, we define $ N_p $ as  
$$
N_p := \left( \int_{-p/2}^{p/2} \left| \operatorname{rect}^{\{p\}}(x) \right|^2 dx \right)^{-1}.
$$  
This ensures that the functions $ \{\phi_{\gamma,p}(x)\}_{\gamma \in \Gamma}$ have unit norm in $ L^2(\mathbb{R}) $ for every $ p \ge 1 $, regardless of the parameters $ u,\xi \in \mathbb{R} $ and $ t > 0 $.

\subsubsection{Some useful Lemmas}
We begin this section by stating a well-known lemma from the literature on B-splines.
\begin{Lemma}[Properties of the $ p $-fold Convolution, \cite{schoenberg1946contributions,schoenberg1969cardinal,de2018stability}]
\label{lemma:schoenberg}
The following properties hold:
\begin{itemize}
    \item[(i)] For every $ p \ge 1 $, the function $ \operatorname{rect}^{\{p\}}(x) $ is even and compactly supported in the interval $ [-\frac{p}{2}, \frac{p}{2}] $.
    
    \item[(ii)] For every $ p \ge 2 $, the convolution admits the explicit formula:
    \begin{equation}
        \operatorname{rect}^{\{p\}}(x) = \frac{1}{(p-1)!} \sum_{j=0}^{p} (-1)^j \binom{p}{j} \left( \max\{0, x + \tfrac{p}{2} - j \} \right)^{p-1}.
        \label{formulaconvoluzione}
    \end{equation}
\end{itemize}
\end{Lemma}

We now introduce the Fourier transform of the dictionary \eqref{dizionarioconvoluzionesupvar}.

\begin{Lemma}[Fourier Transform of $ \phi_{\gamma,p} $]
    Let $ \phi_{\gamma,p}(x) $ be the wavelet atom defined by \eqref{dizionarioconvoluzionesupvar} or its compact support form. Then, the Fourier transform of $ \phi_{\gamma,p} $ is given by:
$$
\hat{\phi}_{\gamma,p}(\omega) = \frac{e^{i p (2\pi \xi - \omega) u}}{t^{\frac{2p - 1}{2}}} \left( \frac{\sin\left( \frac{(2\pi \xi - \omega) t}{2} \right)}{ \frac{(2\pi \xi - \omega)}{2} } \right)^p, \quad \forall p \ge 1.
$$

\end{Lemma}
\begin{proof}
We proceed by induction on $ p \ge 1 $.

Let's start with the base case: $ p = 1 $. We recall that:
$$
\phi_{\gamma,1}(x) = \sqrt{\frac{1}{t}} \, \operatorname{rect}\left( \frac{x - u}{t} \right) e^{2\pi i \xi x} = \sqrt{\frac{1}{t}} \, \mathds{1}_{[u - \frac{t}{2}, \, u + \frac{t}{2}]}(x) \, e^{2\pi i \xi x}.
$$
Its Fourier transform is computed as:
$$
\hat{\phi}_{\gamma,1}(\omega) = \sqrt{\frac{1}{t}} \int_{u - \frac{t}{2}}^{u + \frac{t}{2}} e^{2\pi i \xi x} e^{-i \omega x} dx = \sqrt{\frac{1}{t}} \int_{u - \frac{t}{2}}^{u + \frac{t}{2}} e^{i(2\pi \xi - \omega)x} dx.
$$
Letting $ \tilde{\omega} := 2\pi \xi - \omega $, we get:
$$
\hat{\phi}_{\gamma,1}(\omega) = \sqrt{\frac{1}{t}} \cdot \frac{2 \sin\left( \frac{\tilde{\omega} t}{2} \right)}{\tilde{\omega}} \cdot e^{i \tilde{\omega} u} = \frac{e^{i(2\pi \xi - \omega) u}}{t^{1/2}} \cdot \frac{ \sin\left( \frac{(2\pi \xi - \omega)t}{2} \right) }{ \frac{(2\pi \xi - \omega)}{2} }.
$$
This matches the claimed formula for $ p = 1 $.

Inductive step: assume true for $ p \ge 1 $, prove for $ p+1 $

By the definition of $ \operatorname{rect}^{\{p+1\}} $, we have:
$$
\phi_{\gamma,p+1}(x) = \sqrt{\frac{N_{p+1}}{t}} \left( \operatorname{rect}^{\{p\}} * \operatorname{rect} \right)\left( \frac{x - u}{t} \right) e^{2\pi i \xi x}.
$$
Using the convolution property and the fact that both functions are scaled by $ t $, we write:
$$
\phi_{\gamma,p+1}(x) = \left( \phi_{\gamma,p} * \phi_{\gamma,1} \right)(x),
$$
up to normalization. Since convolution in the time domain corresponds to multiplication in the frequency domain, we obtain:
$$
\hat{\phi}_{\gamma,p+1}(\omega) = \hat{\phi}_{\gamma,p}(\omega) \cdot \hat{\phi}_{\gamma,1}(\omega).
$$
By the inductive hypothesis and the base case:
$$
\hat{\phi}_{\gamma,p}(\omega) = \frac{e^{i p (2\pi \xi - \omega) u}}{t^{\frac{2p - 1}{2}}} \left( \frac{\sin\left( \frac{(2\pi \xi - \omega)t}{2} \right)}{ \frac{(2\pi \xi - \omega)}{2} } \right)^p,
$$
$$
\hat{\phi}_{\gamma,1}(\omega) = \frac{e^{i (2\pi \xi - \omega) u}}{t^{1/2}} \left( \frac{\sin\left( \frac{(2\pi \xi - \omega)t}{2} \right)}{ \frac{(2\pi \xi - \omega)}{2} } \right).
$$
Multiplying the two, we get:
$$
\hat{\phi}_{\gamma,p+1}(\omega) = \frac{e^{i(p+1)(2\pi \xi - \omega) u}}{t^{\frac{2(p+1) - 1}{2}}} \left( \frac{\sin\left( \frac{(2\pi \xi - \omega)t}{2} \right)}{ \frac{(2\pi \xi - \omega)}{2} } \right)^{p+1},
$$
which concludes the inductive proof.

\end{proof}

Now, we state a key lemma and the Uncertainty Theorem for the wavelet dictionary defined in~\eqref{dizionarioconvoluzionesupvar}.

Recall from the third statement of Lemma \ref{lemma:uncertaintyprop} that it is sufficient to compute the uncertainty of the function $\operatorname{rect}^{\{p\}}(x)$, which is equivalent for $p \ge 2$ to the function associated with its dictionary. In fact:
\begin{equation}
    U(\mathcal{C}_{\gamma,p})=U(\operatorname{rect}^{\{p\}}).
\end{equation}
This holds because the uncertainty functional is invariant under dilations, translations, and multiplication by unimodular complex exponentials, as established in Lemma~\ref{lemma:uncertaintyprop}.

We will now provide an upper and lower bound estimate for the uncertainty functional for each fixed $p$, and finally, we will compute the limit as $p \to +\infty$.

The uncertainty of the function $rect^{{p}}(x)$ for every $p \ge 2$ is:
\begin{equation}
    U(rect^{\{p\}})=\displaystyle\frac{\left(\displaystyle\int_{\mathbb{R}} \omega^2\left(\displaystyle\frac{\sin{\displaystyle
    \left(\frac{\omega}{2}\right)}}{\left(\displaystyle
    \frac{\omega}{2}\right)}\right)^{2p} d\omega\right )\left(\displaystyle\int_{-\frac{p}{2}}^{\frac{p}{2}}x^2|rect^{\{p\}}(x)|^2 dx\right)}{\left(\displaystyle\int_{\mathbb{R}} \left(\displaystyle\frac{\sin{\displaystyle
    \left(\frac{\omega}{2}\right)}}{\left(\displaystyle
    \frac{\omega}{2}\right)}\right)^{2p} d\omega\right )\left(\displaystyle\int_{-\frac{p}{2}}^{\frac{p}{2}}|rect^{\{p\}}(x)|^2 dx\right)}.
    \label{incertezzacorretta rect}
\end{equation}

For convenience, we denote by $\nu_p$ and $u_p$ the variance of the density amplitude $\left(\displaystyle\frac{\sin{\displaystyle
    \left(\frac{\omega}{2}\right)}}{\left(\displaystyle
    \frac{\omega}{2}\right)}\right)^{p} $ and $rect^{\{p\}}(x)$, respectively, as follows:
 \begin{equation}
        \nu_p:=\displaystyle\frac{\displaystyle\int_{\mathbb{R}} \omega^2\left(\displaystyle\frac{\sin{\displaystyle
    \left(\frac{\omega}{2}\right)}}{\left(\displaystyle
    \frac{\omega}{2}\right)}\right)^{2p} d\omega}{\displaystyle\int_{\mathbb{R}} \left(\displaystyle\frac{\sin{\displaystyle
    \left(\frac{\omega}{2}\right)}}{\left(\displaystyle
    \frac{\omega}{2}\right)}\right)^{2p} d\omega}; \quad 
    u_p:=\displaystyle\frac{\displaystyle\int_{-\frac{p}{2}}^{\frac{p}{2}}x^2|rect^{\{p\}}(x)|^2 dx}{\displaystyle\int_{-\frac{p}{2}}^{\frac{p}{2}}|rect^{\{p\}}(x)|^2 dx}.
    \end{equation} 
%the variance of the density amplitude $rect^{\{p\}}(x)$.

%We now conclude by showing that the function $h(p) := U(\text{rect}^{\{p\}})$, defined for $p \ge 2$, is decreasing. Therefore, the limit is the infimum, and the maximum value is $\frac{3}{10}$, which is achieved by the function $\text{rect}^{\{2\}}(x)$, as studied in the article by \cite{mazzoccoli2024refining}.

\begin{Lemma}
\label{lemma:decreasing}
    For every $p \ge 2$, the function $U(\text{rect}^{\{p\}})$ is decreasing in $p$.
\end{Lemma}

\begin{proof}
    Noticing from the upper and lower bound estimate \eqref{disinceetezza} that the function $U(\text{rect}^{\{p\}})$ is contained between two decreasing functions, it will thus decrease monotonically or oscillate.

Suppose, for the sake of contradiction, that the function $U(\text{rect}^{\{p\}})$ oscillates for $p \ge 2$. Then, since it is the product of the functions $\nu_p$ and $u_p$, at least one of the two functions must oscillate.

Since the function $\nu_p$ is the ratio of the two functions $k_1(p):=\displaystyle\int_{\mathbb{R}}\left(\displaystyle\frac{\sin(\frac{\omega}{2})}{\frac{\omega}{2}}\right)^{2p}d\omega$ and $k_2(p):=\displaystyle\int_{\mathbb{R}}\omega^2\left(\displaystyle\frac{\sin(\frac{\omega}{2})}{\frac{\omega}{2}}\right)^{2p}d\omega$, we would have that at least one of the two functions, $k_1(p)$ and $k_2(p)$, must oscillate, but both functions are decreasing. In fact, for every $p \ge 3$:
\begin{equation}
\begin{split}
    k_1(p)-k_1(p-1)&=\displaystyle\int_{\mathbb{R}}\left[\left(\displaystyle\frac{\sin(\frac{\omega}{2})}{\frac{\omega}{2}}\right)^{2p}-\left(\displaystyle\frac{\sin(\frac{\omega}{2})}{\frac{\omega}{2}}\right)^{2(p-1)}\right]d\omega\\&=\displaystyle\int_{\mathbb{R}}\left(\displaystyle\frac{\sin(\frac{\omega}{2})}{\frac{\omega}{2}}\right)^{2p}\left[1-\left(\displaystyle\frac{\frac{\omega}{2}}{\sin{(\frac{\omega}{2}})}\right)^2\right]d\omega \le 0
    \end{split}
\end{equation}
since $\left(\displaystyle\frac{\sin(\frac{\omega}{2})}{\frac{\omega}{2}}\right)^{2p}\left[1-\left(\displaystyle\frac{\frac{\omega}{2}}{\sin{(\frac{\omega}{2}})}\right)^2\right] \le 0$  $\forall \omega \in \mathbb{R}, p \ge 2$.

Similarly, using a similar procedure, it is obtained that $k_2(p) - k_2(p-1) \le 0$ for every $p \ge 3$.

Therefore, $k_1(p)$ and $k_2(p)$ are monotonic decreasing functions for $p \ge 2$ and non-oscillatory; this implies that the function $\nu_p$ is not oscillatory.

The function $u_p$, being the variance of the probability amplitude $rect^{{p}}(x)$ (a non-oscillatory function), is decreasing.

In conclusion, the product of two non-oscillatory functions is non-oscillatory, i.e., $U(\text{rect}^{\{p\}})$ does not oscillate but decreases.
\end{proof}

\subsection{Proof of Theorem \ref{th:Urectp}}

\begin{proof}
    First, let us begin by studying the behavior of $\nu_p$. Thus, we start by determining the asymptotic behavior of the following integral
    \begin{equation}
    \displaystyle\int_{\mathbb{R}} \left(\displaystyle\frac{\sin{\displaystyle
    \left(\frac{\omega}{2}\right)}}{\left(\displaystyle
    \frac{\omega}{2}\right)}\right)^{2p} d\omega.
\end{equation}

Using the Watson/Laplace technique, as shown in the article by \cite{schlage2020asymptotic}, and expanding via Taylor series, we obtain that:
\begin{equation}
    \displaystyle\int_{\mathbb{R}} \left(\displaystyle\frac{\sin{\displaystyle
    \left(\frac{\omega}{2}\right)}}{\left(\displaystyle
    \frac{\omega}{2}\right)}\right)^{2p} d\omega=4\displaystyle\int_{\mathbb{R}} \left(\displaystyle\frac{\sin{\omega}}{\omega}\right)^{2p} d\omega \; \sim \; 2\displaystyle\sqrt{\frac{3\pi}{p}}.
\end{equation}

It then follows that, since this integral is a continuous function of $p$ (a sequence is pointwise continuous), and its asymptotic behavior in the limit is given by the function $2\sqrt{\frac{3\pi}{p}}$, there exist two constants $c_2, c_1 \in \mathbb{R} \setminus \{0\}$, with $c_2 > c_1$, such that:
\begin{equation}
      2\displaystyle\sqrt{\frac{3\pi}{p+c_2}} \le \displaystyle\int_{\mathbb{R}} \left(\displaystyle\frac{\sin{\displaystyle
    \left(\frac{\omega}{2}\right)}}{\left(\displaystyle
    \frac{\omega}{2}\right)}\right)^{2p} d\omega \le 2\displaystyle\sqrt{\frac{3\pi}{p+c_1}}.
    \label{disdis}
\end{equation}

Similarly, by again applying the estimates from the article \cite{schlage2020asymptotic}, and using the Watson/Laplace technique together with a Taylor expansion, we obtain:
\begin{equation}
     \displaystyle\int_{\mathbb{R}} \omega^2\left(\displaystyle\frac{\sin{\displaystyle
    \left(\frac{\omega}{2}\right)}}{\left(\displaystyle
    \frac{\omega}{2}\right)}\right)^{2p} d\omega=8\displaystyle\int_{\mathbb{R}} \omega^2\left(\displaystyle\frac{\sin{\omega}}{\omega}\right)^{2p} d\omega \; \sim \; 12\displaystyle\frac{\sqrt{3\pi}}{p^{\frac{3}{2}}}.
\end{equation}

It follows that there exist two constants $c_4, c_3 \in \mathbb{R} \setminus \{0\}$, with $c_4 > c_3$, such that: 

\begin{equation}
12\displaystyle\frac{\sqrt{3\pi}}{(p+c_4)^{\frac{3}{2}}} \le \displaystyle\int_{\mathbb{R}} \omega^2\left(\displaystyle\frac{\sin{\displaystyle
    \left(\frac{\omega}{2}\right)}}{\left(\displaystyle
    \frac{\omega}{2}\right)}\right)^{2p} d\omega \le 12\displaystyle\frac{\sqrt{3\pi}}{(p+c_3)^{\frac{3}{2}}},
     \label{disomega2}
\end{equation}
\textcolor{black}{and thus, for every $p \ge 2$, we achieve \eqref{disrapportofrequenze}.}

Now, we study the behavior of $u_p$.

In the meantime, using Plancherel's theorem, we observe that:
\begin{equation}
    \displaystyle\int_{-\frac{p}{2}}^{\frac{p}{2}} |rect^{\{p\}}(x)|^2 dx=\left\| rect^{\{p\}}\right\|^2= \displaystyle\frac{1}{2\pi}\left\|\widehat{rect^{\{p\}}}\right\|^2=\displaystyle\frac{1}{2\pi}\int_{\mathbb{R}}\left(\displaystyle\frac{\sin(\frac{\omega}{2})}{\frac{\omega}{2}}\right)^{2p} d\omega ,
    \label{disomega2}
\end{equation}

and thus, using \eqref{disdis}, there exist two constants $d_1, d_2 \in \mathbb{R} \setminus \{0\}$, with $d_2 > d_1$, such that the estimate holds:

\begin{equation}
    \displaystyle\frac{1}{\pi}\displaystyle\sqrt{\frac{3\pi}{p+d_2}} \le  \displaystyle\int_{-\frac{p}{2}}^{\frac{p}{2}} |rect^{\{p\}}(x)|^2 dx \le \displaystyle\frac{1}{\pi}\displaystyle\sqrt{\frac{3\pi}{p+d_1}}.
    \label{disspazio importnte}
\end{equation}

As for the last integral to be estimated, we can again use Plancherel's theorem and the properties of the Fourier transform:
\begin{equation}
    \displaystyle\int_{-\frac{p}{2}}^{\frac{p}{2}} x^2|rect^{\{p\}}(x)|^2 dx=\left\|x \cdot rect^{\{p\}}\right\|^2=\displaystyle\frac{1}{2\pi}\left\|\widehat{x \cdot rect^{\{p\}}}\right\|^2=\displaystyle\frac{1}{2\pi}\left\|i\partial_{\omega}(\widehat{rect^{\{p\}}})\right\|^2,
\end{equation}

where
\begin{equation}
   \left\|i\partial_{\omega}(\widehat{rect^{\{p\}}})\right\|^2=\displaystyle\frac{p^2}{4}\displaystyle\int_{\mathbb{R}}\left(\displaystyle\frac{\sin(\frac{\omega}{2})}{\frac{\omega}{2}}\right)^{2p}\left[\displaystyle\frac{\frac{\omega}{2}\cos(\frac{\omega}{2})-\sin(\frac{\omega}{2})}{(\frac{\omega
   }{2})^2}\right]^2 d\omega,
\end{equation}

and thus, the equality holds:
\begin{equation}
    \displaystyle\int_{-\frac{p}{2}}^{\frac{p}{2}} x^2|rect^{\{p\}}(x)|^2 dx=\displaystyle\frac{p^2}{8\pi}\displaystyle\int_{\mathbb{R}}\left(\displaystyle\frac{\sin(\frac{\omega}{2})}{\frac{\omega}{2}}\right)^{2p}\left[\displaystyle\frac{\frac{\omega}{2}\cos(\frac{\omega}{2})-\sin(\frac{\omega}{2})}{(\frac{\omega
   }{2})^2}\right]^2 d\omega.
\end{equation}

Now, using the Watson/Laplace technique again, we obtain that:
\begin{equation}
    \displaystyle\frac{p^2}{8\pi}\displaystyle\int_{\mathbb{R}}\left(\displaystyle\frac{\sin(\frac{\omega}{2})}{\frac{\omega}{2}}\right)^{2p}\left[\displaystyle\frac{\frac{\omega}{2}\cos(\frac{\omega}{2})-\sin(\frac{\omega}{2})}{(\frac{\omega
   }{2})^2}\right]^2 d\omega \; \sim \; \displaystyle\frac{1}{\pi}\frac{\sqrt{3\pi p}}{24}.
\end{equation}

And thus, by making analogous arguments to the previous ones, it follows that for every $p \ge 2$, there exist two constants $d_3, d_4 \in \mathbb{R} \setminus \{0\}$, with $d_4 > d_3$, such that: 
\begin{equation}
\displaystyle\frac{1}{\pi}\frac{\sqrt{3\pi (p+d_3)}}{24} \le \displaystyle\int_{-\frac{p}{2}}^{\frac{p}{2}}x^2|rect^{\{p\}}(x)|^2 dx \le \displaystyle\frac{1}{\pi}\frac{\sqrt{3\pi (p+d_4)}}{24}.
\label{varianzastimaimportante}
\end{equation}
\textcolor{black}{Thus, by combining \eqref{disspazio importnte} and \eqref{varianzastimaimportante},  we obtain \eqref{disrapportospaziale}.}

\textcolor{black}{Finally, by combining \eqref{disrapportofrequenze} and \eqref{disrapportospaziale}, we obtain \eqref{disinceetezza}. And thus, it immediately follows that:}
\begin{equation}
    \displaystyle\lim_{p \to +\infty} U(rect^{\{p\}})=\frac{1}{4}.
\end{equation}
By the Lemma \ref{lemma:decreasing}, since $U(rect^{\{p\}})$ is monotonically decreasing in p, it follows that $U(rect^{\{p\}})$ attains its infimum. 
\end{proof}

\textcolor{black}{
\begin{remark}[On attainability]
    The constants $c_i, d_i$ in (\ref{disrapportofrequenze})-(\ref{disinceetezza}) are $p$-independent and arise from two-sided Watson/Laplace estimates. Therefore, the bounds are generally strict for finite $p$; equality in (\ref{disinceetezza}) is only achieved in the limit $p \rightarrow \infty$, where $U\left(\right.$ rect $\left.^{\{p\}}\right) \rightarrow \frac{1}{4}$. Non-emptiness of the chains is ensured by $c_i, d_i>-2$ and by the orderings $\frac{p+c_2}{\left(p+c_4\right)^3} \leq \frac{p+c_1}{\left(p+c_3\right)^3},\left(p+d_2\right)\left(p+d_3\right) \leq\left(p+d_1\right)\left(p+d_4\right)$, which together imply (\ref{disinceetezza}).
\end{remark}}

\section{Conclusions and Outlook}\label{sec:conclusions}

In this work we introduced a hierarchy of function classes on a fixed compact
interval~$K=[-a,a]$, defined uncertainty operators tailored to them, and
showed that
\begin{enumerate}[label=(\roman*)]
  \item the uncertainty product is \emph{scale--translation invariant}
        (Lemma~\ref{lemma:uncertaintyprop});
  \item its infimum over the asymmetric class $\mathcal{F}^+_{0}(K)$ is already
        attained inside the even sub‑class $\mathcal{P}^+_{0}(K)$
        (Theorem~\ref{th:infimumevenf});
  \item inside two concrete wavelet dictionaries,
        $\{\mathcal G_{\gamma,n}\}_{n\in\mathbb N}$ and
        $\{\mathcal F_{\gamma,n}\}_{n\in\mathbb N}$,
        the \emph{tent function} ($n=1$) uniquely minimises the
        time–frequency uncertainty, with value $U(\mathcal G_{\gamma,1})=U(F_{\gamma,1})=\tfrac{3}{10}$ forall $\gamma \in \Gamma$;
  \item in the family of $p$‑fold self‑convolutions of the rectangle,
        $\operatorname{rect}^{\{p\}}$, the uncertainty decreases
        monotonically to the Heisenberg bound $\tfrac14$ as
        $p\to\infty$ (Section~\ref{sec3}).
\end{enumerate}

%These results give a unified explanation of several empirical observations that had appeared independently in the literature on adaptive wavelet design and Gabor‐frame stability.  They also suggest a principled recipe for building dictionaries with \emph{provably optimal} joint localisation: start from mother atoms that minimise $U$ in the simplest symmetric class, then generate frames by the standard rescale–translate–modulate actions.

\textcolor{black}{\paragraph{Connections and differences.}
Our approach is \emph{invariance-aware} and \emph{constructive}: within the two wavelet dictionaries
we obtain an explicit minimizer (the tent, with $U=\tfrac{3}{10}$) and a smooth path to the
Heisenberg limit via $\mathrm{rect}^{\{p\}}$ (with $U\downarrow\tfrac{1}{4}$ as $p\to\infty$).
This complements Balian–Low–type impossibility statements and periodic uncertainty constants by providing
concrete extremal constructions and a unifying variational lens, and it interfaces with frame-set and
sampling viewpoints as well as recent generalizations of uncertainty principles.
}

\textcolor{black}{\paragraph{Perspectives and open questions}
The analysis opens a number of directions that we have not addressed here:}
\begin{itemize}
  \item \textbf{Uniqueness beyond evenness.}
        Is the tent profile the \emph{only} global minimiser of $U$ in
        $\mathcal{F}^+_{0}(K)$, up to the invariances of
        Lemma~\ref{lemma:uncertaintyprop}, or could other non‑polynomial
        shapes attain the same value?
  \item \textbf{Higher–dimensional extensions.}
        How do the minimisation results change in $\mathbb R^d$?
        In particular, does radial symmetry play the same role played by
        evenness in~$d=1$?
  \item \textbf{Alternative measures of spread.}
        Replacing the $L^2$‑based variance with entropic or $L^p$
        moments may lead to different optimal shapes; can one classify them?
  \item \textbf{Optimal frames versus optimal atoms.}
        Does minimising $U$ for the single atom always translate into the
        most stable (or sparsest) frame when the whole dictionary is
        generated, or can global frame‐level criteria override atom‐level
        optimality?
%  \item \textbf{Connections with operator theory.} Can the present framework be reformulated in terms of spectral        properties of suitable localisation operators, providing new         insights into their eigenfunctions and spectra?
\end{itemize}

\section*{Acknowledgments}
P. Vellucci is a member of the INdAM Research group GNCS. We are sincerely grateful to Prof. Laura De Carli for her valuable support and insightful guidance.

\bibliographystyle{model1b-num-names} 
\bibliography{example}

\begin{thebibliography}{27}
\expandafter\ifx\csname natexlab\endcsname\relax\def\natexlab#1{#1}\fi
\providecommand{\bibinfo}[2]{#2}
\ifx\xfnm\relax \def\xfnm[#1]{\unskip,\space#1}\fi
%Type = Article
\bibitem[{Battle(1988)}]{battle1988heisenberg}
\bibinfo{author}{G.~Battle},
\newblock \bibinfo{title}{Heisenberg proof of the balian-low theorem},
\newblock \bibinfo{journal}{Letters in Mathematical Physics} \bibinfo{volume}{15} (\bibinfo{year}{1988}) \bibinfo{pages}{175--177}.
%Type = Article
\bibitem[{Caragea et~al.(2023)Caragea, Lee, Philipp, and Voigtlaender}]{caragea2023balian}
\bibinfo{author}{A.~Caragea}, \bibinfo{author}{D.~G. Lee}, \bibinfo{author}{F.~Philipp}, \bibinfo{author}{F.~Voigtlaender},
\newblock \bibinfo{title}{A balian--low type theorem for gabor riesz sequences of arbitrary density},
\newblock \bibinfo{journal}{Mathematische Zeitschrift} \bibinfo{volume}{303} (\bibinfo{year}{2023}) \bibinfo{pages}{48}.
%Type = Article
\bibitem[{Caragea et~al.(2021)Caragea, Lee, Philipp, and Voigtlaender}]{caragea2021quantitative}
\bibinfo{author}{A.~Caragea}, \bibinfo{author}{D.~G. Lee}, \bibinfo{author}{F.~Philipp}, \bibinfo{author}{F.~Voigtlaender},
\newblock \bibinfo{title}{A quantitative subspace balian-low theorem},
\newblock \bibinfo{journal}{Applied and Computational Harmonic Analysis} \bibinfo{volume}{55} (\bibinfo{year}{2021}) \bibinfo{pages}{368--404}.
%Type = Article
\bibitem[{Enstad(2020)}]{enstad2020balian}
\bibinfo{author}{U.~Enstad},
\newblock \bibinfo{title}{The balian--low theorem for locally compact abelian groups and vector bundles},
\newblock \bibinfo{journal}{Journal de Math{\'e}matiques Pures et Appliqu{\'e}es} \bibinfo{volume}{139} (\bibinfo{year}{2020}) \bibinfo{pages}{143--176}.
%Type = Article
\bibitem[{Breitenberger(1985)}]{breitenberger1985uncertainty}
\bibinfo{author}{E.~Breitenberger},
\newblock \bibinfo{title}{Uncertainty measures and uncertainty relations for angle observables},
\newblock \bibinfo{journal}{Foundations of Physics} \bibinfo{volume}{15} (\bibinfo{year}{1985}) \bibinfo{pages}{353--364}.
%Type = Article
\bibitem[{Lebedeva(2017)}]{lebedeva2017inequality}
\bibinfo{author}{E.~A. Lebedeva},
\newblock \bibinfo{title}{An inequality for a periodic uncertainty constant},
\newblock \bibinfo{journal}{Applied and Computational Harmonic Analysis} \bibinfo{volume}{42} (\bibinfo{year}{2017}) \bibinfo{pages}{536--549}.
%Type = Article
\bibitem[{Goh and Yeo(2000)}]{goh2000uncertainty}
\bibinfo{author}{S.~S. Goh}, \bibinfo{author}{C.~H. Yeo},
\newblock \bibinfo{title}{Uncertainty products of local periodic wavelets},
\newblock \bibinfo{journal}{Advances in Computational Mathematics} \bibinfo{volume}{13} (\bibinfo{year}{2000}) \bibinfo{pages}{319--333}.
%Type = Article
\bibitem[{Lebedeva and Prestin(2014)}]{lebedeva2014periodic}
\bibinfo{author}{E.~A. Lebedeva}, \bibinfo{author}{J.~Prestin},
\newblock \bibinfo{title}{Periodic wavelet frames and time--frequency localization},
\newblock \bibinfo{journal}{Applied and Computational Harmonic Analysis} \bibinfo{volume}{37} (\bibinfo{year}{2014}) \bibinfo{pages}{347--359}.
%Type = Article
\bibitem[{Ghosh and Selvan(2025{\natexlab{a}})}]{ghosh2025gabor}
\bibinfo{author}{R.~Ghosh}, \bibinfo{author}{A.~A. Selvan},
\newblock \bibinfo{title}{On gabor frames generated by b-splines, totally positive functions, and hermite functions},
\newblock \bibinfo{journal}{Applied Numerical Mathematics} \bibinfo{volume}{207} (\bibinfo{year}{2025}{\natexlab{a}}) \bibinfo{pages}{1--23}.
%Type = Article
\bibitem[{Ghosh and Selvan(2025{\natexlab{b}})}]{ghosh2025obstructions}
\bibinfo{author}{R.~Ghosh}, \bibinfo{author}{A.~A. Selvan},
\newblock \bibinfo{title}{Obstructions for gabor frames of the second-order b-spline},
\newblock \bibinfo{journal}{Advances in Computational Mathematics} \bibinfo{volume}{51} (\bibinfo{year}{2025}{\natexlab{b}}) \bibinfo{pages}{27}.
%Type = Article
\bibitem[{Aldahleh et~al.(2025)Aldahleh, Iosevich, Iosevich, Jaimangal, Mayeli, and Pack}]{aldahleh2025additive}
\bibinfo{author}{K.~Aldahleh}, \bibinfo{author}{A.~Iosevich}, \bibinfo{author}{J.~Iosevich}, \bibinfo{author}{J.~Jaimangal}, \bibinfo{author}{A.~Mayeli}, \bibinfo{author}{S.~Pack},
\newblock \bibinfo{title}{Additive energy, uncertainty principle and signal recovery mechanisms},
\newblock \bibinfo{journal}{arXiv preprint arXiv:2504.14702}  (\bibinfo{year}{2025}).
%Type = Article
\bibitem[{Selvan and Radha(2017)}]{selvan2017optimal}
\bibinfo{author}{A.~A. Selvan}, \bibinfo{author}{R.~Radha},
\newblock \bibinfo{title}{An optimal result for sampling density in shift-invariant spaces generated by meyer scaling function},
\newblock \bibinfo{journal}{Journal of Mathematical Analysis and Applications} \bibinfo{volume}{451} (\bibinfo{year}{2017}) \bibinfo{pages}{197--208}.
%Type = Article
\bibitem[{Selvan and Radha(2016)}]{selvan2016sampling}
\bibinfo{author}{A.~A. Selvan}, \bibinfo{author}{R.~Radha},
\newblock \bibinfo{title}{Sampling and reconstruction in shift invariant spaces of b-spline functions},
\newblock \bibinfo{journal}{Acta Applicandae Mathematicae} \bibinfo{volume}{145} (\bibinfo{year}{2016}) \bibinfo{pages}{175--192}.
%Type = Article
\bibitem[{Abreu and Speckbacher(2025)}]{abreu2025donoho}
\bibinfo{author}{L.~D. Abreu}, \bibinfo{author}{M.~Speckbacher},
\newblock \bibinfo{title}{Donoho-logan large sieve principles for the wavelet transform},
\newblock \bibinfo{journal}{Applied and Computational Harmonic Analysis} \bibinfo{volume}{74} (\bibinfo{year}{2025}) \bibinfo{pages}{101709}.
%Type = Article
\bibitem[{Dades and Daher(2024)}]{dades2024heisenberg}
\bibinfo{author}{A.~Dades}, \bibinfo{author}{R.~Daher},
\newblock \bibinfo{title}{Heisenberg and donoho stark uncertainty principles associated to the mehler--fock wavelet transform},
\newblock \bibinfo{journal}{International Journal of Applied and Computational Mathematics} \bibinfo{volume}{10} (\bibinfo{year}{2024}) \bibinfo{pages}{158}.
%Type = Article
\bibitem[{Dades and Daher(2025)}]{dades2025new}
\bibinfo{author}{A.~Dades}, \bibinfo{author}{R.~Daher},
\newblock \bibinfo{title}{New uncertainty principles for the free metaplectic wavelet transform},
\newblock \bibinfo{journal}{Integral Transforms and Special Functions}  (\bibinfo{year}{2025}) \bibinfo{pages}{1--14}.
%Type = Article
\bibitem[{Dar and Bhat(2024)}]{dar2024n}
\bibinfo{author}{A.~H. Dar}, \bibinfo{author}{M.~Y. Bhat},
\newblock \bibinfo{title}{N-dimensional wave packet transform and associated uncertainty principles in the free metaplectic transform domain},
\newblock \bibinfo{journal}{Mathematical Methods in the Applied Sciences} \bibinfo{volume}{47} (\bibinfo{year}{2024}) \bibinfo{pages}{13199--13220}.
%Type = Article
\bibitem[{Wang and Zheng(2024)}]{wang2024benedicks}
\bibinfo{author}{X.~Wang}, \bibinfo{author}{S.~Zheng},
\newblock \bibinfo{title}{On benedicks--amrein--berthier uncertainty principles for continuous quaternion wavelet transform},
\newblock \bibinfo{journal}{Mathematical Methods in the Applied Sciences} \bibinfo{volume}{47} (\bibinfo{year}{2024}) \bibinfo{pages}{13467--13484}.
%Type = Article
\bibitem[{Rivero and Vellucci(2023)}]{rivero2023solution}
\bibinfo{author}{J.~A. Rivero}, \bibinfo{author}{P.~Vellucci},
\newblock \bibinfo{title}{A solution for the greedy approximation of a step function with a waveform dictionary},
\newblock \bibinfo{journal}{Communications in Nonlinear Science and Numerical Simulation} \bibinfo{volume}{116} (\bibinfo{year}{2023}) \bibinfo{pages}{106890}.
%Type = Book
\bibitem[{Rudin(1987)}]{rudin1987real}
\bibinfo{author}{W.~Rudin}, \bibinfo{title}{Real and complex analysis}, \bibinfo{publisher}{McGraw-Hill, Inc.}, \bibinfo{year}{1987}.
%Type = Article
\bibitem[{Schoenberg(1946)}]{schoenberg1946contributions}
\bibinfo{author}{I.~J. Schoenberg},
\newblock \bibinfo{title}{Contributions to the problem of approximation of equidistant data by analytic functions. part b. on the problem of osculatory interpolation. a second class of analytic approximation formulae},
\newblock \bibinfo{journal}{Quarterly of Applied Mathematics} \bibinfo{volume}{4} (\bibinfo{year}{1946}) \bibinfo{pages}{112--141}.
%Type = Article
\bibitem[{Schoenberg(1969)}]{schoenberg1969cardinal}
\bibinfo{author}{I.~Schoenberg},
\newblock \bibinfo{title}{Cardinal interpolation and spline functions},
\newblock \bibinfo{journal}{Journal of Approximation theory} \bibinfo{volume}{2} (\bibinfo{year}{1969}) \bibinfo{pages}{167--206}.
%Type = Article
\bibitem[{De~Carli and Vellucci(2018)}]{de2018stability}
\bibinfo{author}{L.~De~Carli}, \bibinfo{author}{P.~Vellucci},
\newblock \bibinfo{title}{Stability results for gabor frames and the p-order hold models},
\newblock \bibinfo{journal}{Linear Algebra and its Applications} \bibinfo{volume}{536} (\bibinfo{year}{2018}) \bibinfo{pages}{186--200}.
%Type = Article
\bibitem[{Mazzoccoli et~al.(2024)Mazzoccoli, Rivero, and Vellucci}]{mazzoccoli2024refining}
\bibinfo{author}{A.~Mazzoccoli}, \bibinfo{author}{J.~A. Rivero}, \bibinfo{author}{P.~Vellucci},
\newblock \bibinfo{title}{Refining heisenberg’s principle: A greedy approximation of step functions with triangular waveform dictionaries},
\newblock \bibinfo{journal}{Mathematics and Computers in Simulation} \bibinfo{volume}{225} (\bibinfo{year}{2024}) \bibinfo{pages}{165--176}.
%Type = Article
\bibitem[{Busch et~al.(2007)Busch, Heinonen, and Lahti}]{busch2007heisenberg}
\bibinfo{author}{P.~Busch}, \bibinfo{author}{T.~Heinonen}, \bibinfo{author}{P.~Lahti},
\newblock \bibinfo{title}{Heisenberg's uncertainty principle},
\newblock \bibinfo{journal}{Physics reports} \bibinfo{volume}{452} (\bibinfo{year}{2007}) \bibinfo{pages}{155--176}.
%Type = Book
\bibitem[{Mallat(1999)}]{mallat1999wavelet}
\bibinfo{author}{S.~Mallat}, \bibinfo{title}{A wavelet tour of signal processing}, \bibinfo{publisher}{Elsevier}, \bibinfo{year}{1999}.
%Type = Article
\bibitem[{Schlage-Puchta(2020)}]{schlage2020asymptotic}
\bibinfo{author}{J.-C. Schlage-Puchta},
\newblock \bibinfo{title}{Asymptotic evaluation of $\int_{0}^{\infty}\left (\frac{sin x}{x}\right)^n dx$},
\newblock \bibinfo{journal}{arXiv preprint arXiv:2010.11759}  (\bibinfo{year}{2020}).

\end{thebibliography}



%% For example:

%% \bibitem[Aladro et al.(2015)]{Aladro15} Aladro, R., Martín, S., Riquelme, D., et al. 2015, \aas, 579, A101

%%\end{thebibliography}

\end{document}